\documentclass[article,15pt,dvipsnames]{article}
\usepackage{amsmath,amssymb}
\usepackage{amsthm}
\usepackage{mathtools}
\usepackage{graphics,epsfig,color}
\usepackage[mathscr]{euscript}
\usepackage{thmtools}
\usepackage[top=2cm]{geometry}
\usepackage{comment}
\usepackage{enumerate}

\usepackage{graphicx}
\usepackage{graphics}
\usepackage{caption}
\usepackage{subfig}
\usepackage{hyperref}
\usepackage{tikz}
\usetikzlibrary{arrows}

\usepackage{tabularx}
\usepackage{bm}
\usepackage{multirow}
\usepackage{pifont}

\usepackage{bbold}

\newcommand{\de}{\mathrm{d}}

\newcommand{\upperRomannumeral}[1]{\uppercase\expandafter{\romannumeral#1}}
\newcommand{\lowerRomannumeral}[1]{\lowercase\expandafter{\romannumeral#1}}

\theoremstyle{plain}
\newtheorem{theorem}{Theorem}[section]

\newtheorem{lemma}[theorem]{Lemma}
\newtheorem{proposition}[theorem]{Proposition}

\theoremstyle{definition}
\newtheorem{definition}[theorem]{Definition}
\newtheorem{assumption}[theorem]{Assumption}

\theoremstyle{remark}
\newtheorem{remark}[theorem]{Remark}

\newenvironment{myproof}[1] {{
		
		\noindent\it{Proof of {#1}}. }}{\hfill\qedsymbol}
\makeatletter
\def\ps@pprintTitle{ 
	\let\@oddhead\@empty
	\let\@evenhead\@empty
	\def\@oddfoot{\footnotesize\itshape
		\ifx\@empty\@empty
		\else\@journal\fi\hfill\today}%
	\let\@evenfoot\@oddfoot
}

\newcommand{\kws}{\bigskip\par\addvspace\medskipamount{\rightskip=0pt plus1cm}{\noindent \bfseries Keywords: \enspace}}

\newcommand{\MSC}{\bigskip\par\addvspace\medskipamount{\rightskip=0pt plus1cm}{\noindent \bfseries Mathematics Subject Classification (2020): \enspace}}

\usepackage{natbib}

\date{}

\begin{document}
	
	\title{Stochastic Internal Habit Formation and Optimality\thanks{M.A. is a member of the Gruppo Nazionale per l’Analisi Matematica, la Probabilità e le loro
			Applicazioni (GNAMPA) of the Istituto Nazionale di Alta Matematica (INdAM).
			M.A.,  A.B. and F.G. are supported by the Italian Ministry of University and Research (MUR),
			in the framework of PRIN projects 2017FKHBA8 001 (The Time-Space Evolution of Economic Activities:
			Mathematical Models and Empirical Applications) and 20223PNJ8K (Impact of the Human Activities on the Environment
			and Economic Decision Making in a Heterogeneous Setting: Mathematical Models and Policy Implications).}}
	
	\author{Michele Aleandri\thanks{Department of Economics and Business, University of Florence, Florence (\textit{michele.aleandri@unifi.it})}
		\and Alessandro Bondi\thanks{Department of AI, Data and Decision Sciences, Luiss University Rome (\textit{abondi@luiss.it})}\and 
		Fausto Gozzi\thanks{Department of AI, Data and Decision Sciences, Luiss University Rome 
			(\textit{fgozzi@luiss.it})}}
	\maketitle
	
	\begin{abstract}
		Growth models with internal habit formation have been studied in various settings under the assumption of deterministic dynamics. The purpose of this paper is to explore a stochastic version of the model in \cite{Carroll1997,Carroll2000}, one of the most influential on the subject. The goal is twofold: on the one hand, to determine how far we can advance in the technical study of the model; on the other hand, to assess whether at least some of the deterministic outcomes remain valid in the stochastic setting.
		The resulting optimal control problem proves to be challenging, primarily due to the lack of concavity in the objective function. This feature is present in the model even in the deterministic case (see, e.g., \cite{BAMBI2020}).
		We develop an approach based on dynamic programming to establish several useful results, including the regularity of the solution to the corresponding HJB equation, a verification theorem and the existence of a unique optimal control in feedback form.
	\end{abstract}
	
	\kws Habit formation; Optimal investment and consumption; Stochastic control; Dynamic programming; Viscosity solution; Hamilton–Jacobi–Bellman equation.
	
	\MSC 49L20; 49L12; 49L25; 49J55; 93E20.
	
	
	\section{Introduction}
	In two highly influential works, \cite{Carroll1997,Carroll2000} analyze the
	dynamics of an endogenous growth model incorporating internal habit
	formation, showing that this feature plays a key role in explaining how
	increased economic growth can lead to higher savings. This finding is
	particularly significant because it relies on a ``multiplicative'' habit,
	as opposed to the ``subtractive'' form, which fosters addictive behaviour by
	requiring consumption to consistently exceed the habit stock, potentially
	resulting in infinitely negative utility. Yet the subtractive form dominates the continuous-time stochastic literature
	(see, e.g., \cite{DetempleKaratzas2003,EnglezosKaratzas2009}), where the
	linearity of the habit makes concavity-based methods available.
	
	In \cite{BAMBI2020}, the optimality of the solution in \cite{Carroll2000} is
	revisited via the Dynamic Programming approach, which, unlike the Maximum Principle,
	provides sufficient optimality conditions without concavity, a property that
	fails in this model even in the deterministic case. In stochastic models,
	multiplicative preferences have been studied mainly in portfolio choice
	problems, see for instance \cite{vanBilsen_Bovenberg_Laeven_2020,ABY22,guasoni2020shortfall} and \cite{
		angoshtari2026optimal}. In particular, in \cite{ABY22,guasoni2020shortfall} and \cite{
		angoshtari2026optimal}, homogeneity reduces the problem to one dimension;
	the investment control in a complete market allows the HJB equation to
	be linearized by convex duality and transversality to be obtained through
	the state-price density.  None of these tools is available in a production
	economy with exogenous noise.
	
	This paper investigates a stochastic version of the model in
	\cite{Carroll2000} with independent multiplicative noises on capital and
	habit, and extends to it the Dynamic Programming approach in
	\cite{BAMBI2020}. Besides the lack of concavity, the problem is made delicate
	by the singularity of the utility, which diverges to $-\infty$ as
	consumption vanishes, and by the state constraint on capital.
	The main results that we establish are the following: (i) the value function, denoted by $V$, is continuous (Theorem \ref{prop_cont}); (ii) $V$ is a classical solution of the HJB equation in the open first quadrant $\mathbb{R}^2_{++}$ (Theorem \ref{classical_sol}); 
	(iii) a verification theorem (Theorem \ref{teo:sufficient}), which gives a sufficient condition for optimality; (iv) the existence and uniqueness of an optimal control in feedback form (Theorem \ref{ex_uni_CLE}). In particular, (ii) is obtained in two steps: we first prove that $V$ is a viscosity solution of the HJB equation, see Theorem \ref{prop:VisVisc}, and then apply a local regularization approach based on crucial results by \cite{safonov1989classical} and \cite{crandall1992user}.
	\vspace{2mm}
	
	The paper is organized as follows. Section \ref{sec_model} introduces the model and the optimal control problem we are interested in. In Section \ref{sec_HJB}, the corresponding HJB equation is formally obtained, together with some properties of the maximum value Hamiltonian. Some important properties (such as monotonicity and continuity) of the value function $V$ associated with the control problem are then investigated in Section \ref{sec_regularity}.
	In Section \ref{sec_new}, we prove that $V$ is a classical solution of the HJB equation, see Theorem \ref{classical_sol}. This section is divided into two subsections: in the first (Subsection \ref{sec_newvisc}), we demonstrate that $V$ is a viscosity solution, while in the second (Subsection \ref{sec_classical}), we prove the main result, Theorem \ref{classical_sol}. Finally, a verification theorem and the existence and uniqueness of an optimal control in feedback form are established in Section \ref{sec_verification}. In Appendix \ref{Appendix}, we prove some technical results which are employed in the arguments of Subsection \ref{sec_classical} and Section \ref{sec_verification}.
	
	\section{The Model}\label{sec_model}
	
	Consider a probability space $(\Omega,\mathcal{F},\mathbb{P})$ which supports two independent Brownian motions $W_1$ and $W_2$. We endow $\Omega$ with the minimal augmented filtration $\mathbb{F}=(\mathcal{F}_t)_t$ generated by $W_1$ and $W_2$. The model state variables are the processes  $k=(k_t)_{t\ge0}$ and $h=(h_t)_{t\ge 0}$, representing the capital and the habit, respectively. Their evolution is described by the following controlled linear stochastic differential equations (henceforth, SDEs):
	\begin{align}\label{eq:k}
		&\de k_t = \big(Bk_t-c_t\big)\de t + \beta_1k_t\de W_{1,t},\quad k_0>0,\\\label{eq:h}
		&\de h_t = \rho\big(c_t-h_t\big)\de t + \beta_2h_t\de W_{2,t},\quad h_0>0.
	\end{align}
	Here,    $\beta_1,\beta_2>0$ are the volatilities of capital and habit, $\rho\in(0,1)$ and $B> 0$ is the maximal expected exponential growth rate of the capital stock. The positive initial conditions $k_0$ and $h_0$ are assumed to be deterministic. 
	The control process $c=(c_t)_{t\ge0}$ represents the consumption; it is progressively measurable  and (strictly) positive, with locally integrable trajectories. For any choice of $c$ and $h_0,\,k_0$, the dynamics \eqref{eq:k}-\eqref{eq:h} admit unique solutions, up to indistinguishability. Their explicit expressions are given by 
	\begin{align}\label{expl_eq_k}
		&k_t=\mathcal{E}(Bt+\beta_1 W_{1,t})\left(
		k_0-\int_0^t\frac{c_s}{\mathcal{E}(Bs+\beta_1W_{1,s})}\de s\right),\quad t\ge0,\\
		\label{expl_eq_h}
		&h_t=\mathcal{E}(-\rho t+\beta_2 W_{2,t})\left(
		h_0+\rho \int_0^t\frac{c_s}{\mathcal{E}(-\rho s+\beta_2W_{2,s})}\de s\right),\quad t\ge 0,
	\end{align}
	where $\mathcal{E}(X_t)$ denotes the Doléans-Dade exponential of a continuous semimartingale $X$ evaluated at time $t$,  i.e., $\mathcal{E}(X_t)=\mathcal{E}(X)_t$. 
	In the sequel, when the context requires to clarify the dependence of these solutions on the initial state and the control, we will denote by \begin{equation}\label{notation_dependence}
		h^{h_0,c} \text{ the solution of (\ref{eq:h}) controlled by $c$ starting from $h_0$ at time $0$.}
	\end{equation}
	An analogous notation will be used for the capital process $k^{k_0,c}.$\\
	For a fixed initial condition $(k_0,h_0)$, we denote by $\mathcal{A}(k_0,h_0)$ the set of admissible controls. It  is determined by imposing the positivity state constraints
	\begin{equation}\label{constraints-kh}
		k_t>0 \quad\mbox{ and }\quad h_t>0,\qquad \text{for every } t\geq 0,\,\mathbb{P}-\text{a.s.,}
	\end{equation}
	together with the following relation between $c$ and $k$:
	\begin{equation}\label{bdd_control}
		c\le R\, k,\quad \mathbb{P}\otimes \de t-\text{a.e., for a fixed constant $R>0$ independent of $c$ and $k$.}
	\end{equation}
	We refer the reader to  \cite{magill1981infinite}, \cite{becker1989recursive},    and \cite{balder1993existence} for control constraints similar to \eqref{bdd_control}.
	Summarizing, 
	\begin{equation}
		\mathcal{A}(k_0,h_0):=\left\{c\colon \Omega\times[0,\infty)\to \mathbb{R} \text{ progr. meas.}\, \Bigg| \,\begin{aligned}
			&  \text{$c(\omega,\cdot) \in L^1_\text{loc}(0,\infty),$ for $\mathbb{P}-$a.e. $\omega\in\Omega,$}\\& c>0,\,  \mathbb{P}\otimes \de t-\text{a.e.,  the solutions of } \\&  \text{\eqref{eq:k}-\eqref{eq:h} satisfy \eqref{constraints-kh}}, \text{ and \eqref{bdd_control} holds.}
		\end{aligned}\right\}.\label{def:setofcontrols}
	\end{equation}
	Given $\theta>0,\,\sigma\in(1,\infty)$  and $\gamma\in(0,1]$, we are interested in the objective functional
	\begin{align}\label{def_J}
		J (k_0,h_0;c)&\coloneqq\mathbb{E}\bigg[\int_{0}^\infty \frac{1}{1-\sigma}\left(\frac{c_t}{h_t^\gamma}\right)^{1-\sigma}e^{-\theta t}\de t\bigg]
		\le 0,\quad k_0,\,h_0>0,\,c\in\mathcal{A}(k_0,h_0).
	\end{align}
	From now on, we write $\mathbb{R}_{+}$ for $(0,\infty)$ and $\mathbb{R}^2_{++}$ for $(0,\infty)^2$. The  value function $V\colon \mathbb{R}^2_{++}\to [-\infty,0]$ associated with $J$ in \eqref{def_J} is  \begin{equation}\label{def:valuefunction}
		V(k_0,h_0)=\sup_{c\in\mathcal{A}(k_0,h_0)}  J(k_0,h_0;c),\quad (k_0,\,h_0)\in \mathbb{R}^2_{++}.
	\end{equation}
	
	We assume that the parameters of the model fulfill the following requirement (cf. \cite[Assumption 2.4]{freni2008optimal} and also \cite{freni2006existence}).
	\begin{assumption}\label{assumption1} 
		Suppose that $R\ge B$. Furthermore, the parameters $(\rho,\,\beta_1,\,\beta_2,\,\gamma,\,\sigma,\,\theta)$ satisfy
		\begin{equation}\label{eq_maggiore-minore}
			\theta>(\sigma-1)\bigg(\frac{1}{2}\beta_1^2(1-\gamma)+\gamma\big(\rho+\frac{1}{2}(\beta_2^2-\beta_1^2)\big)^-+\max\left\{2, \frac{1}{\gamma(\sigma-1)}\right\}(\sigma-1)(\beta_1^2+\gamma^2\beta_2^2) \bigg),
		\end{equation}
		where $x^-=\max\{-x,0\}$.
	\end{assumption}
	
	\begin{remark}
		The condition $R\geq B$ is consistent with the general framework where the asymptotic regime $R\to\infty$ is considered. Condition  \eqref{eq_maggiore-minore} ensures that the discount rate $\theta$ dominates the 
		maximal growth of the state variables including their stochastic parts. From a technical point of view, it guarantees the existence of strategies with finite utility, see Proposition \ref{prop:lbound}. 
	\end{remark}

	We conclude this section by showing that  the set of admissible controls is nonempty. 
	\begin{lemma}\label{set_controls}
		For every $(k_0,\,h_0)\in \mathbb{R}^2_{++}$, the set  $\mathcal{A}(k_0,h_0)$ of admissible controls is nonempty. In addition, $\mathcal{A}(k_0,h_0)$ does not depend on $h_0$, that is, 
		$
		\mathcal{A}(k_0,h_0)=\mathcal{A}(k_0).
		$
	\end{lemma}
	\begin{proof}
		Fix $(k_0,h_0)\in\mathbb{R}^2_{++}$. The fact that $\mathcal{A}(k_0,h_0)=\mathcal{A}(k_0)$ follows immediately from \eqref{expl_eq_h}, which implies the positivity of the habit process corresponding to positive controls. Choosing $c=\nu \mathcal{E}(\beta_1 W_1),$ for some constant $\nu>0$ \emph{to be determined}, \eqref{expl_eq_k} reads
		$
		k_t=e^{Bt}\mathcal{E}(\beta_1W_{1,t})
		\left(k_0-\frac{\nu}{B}(1-e^{-Bt})\right).
		$
		Therefore, fixing $\frac{\nu}{B}<k_0$,
		\[
		k_t\ge \left(k_0-\frac{\nu}{B}\right)\mathcal{E}(\beta_1W_{1,t})>0
		,\quad t\ge 0,
		\]
		hence the positivity constraint \eqref{constraints-kh} on $k$ is satisfied. If we further require that $\frac{R}{\nu}(k_0-\frac{\nu}{B})>1$, then
		\[
		R k_t\ge \frac{R}{\nu}\left(k_0-\frac{\nu}{B}\right)c_t>c_t
		,\quad t\ge 0,
		\]
		which proves \eqref{bdd_control}.
		The two previous conditions  are satisfied by $\nu$ sufficiently small. Thus $c\in \mathcal{A}(k_0,h_0)\neq \emptyset$, and  the proof is complete. 
	\end{proof}

	\section{HJB Equation: Formal Derivation}\label{sec_HJB}
	
	Given $(k,h)\in\mathbb{R}^2_{++}$, $\pi=\left[\begin{smallmatrix} \pi_k\\\pi_h \end{smallmatrix}\right]\in\mathbb{R}^2$ and $Q\in\mathbb{R}^{2\times 2}$, the current value Hamiltonian is  
	\begin{equation*}
		H(k,h,\pi,Q;c)= a(k,h;c)^\top \pi+\frac{1}{2}\mathrm{tr}\big(\Sigma(k,h)^\top Q\Sigma(k,h)\big) +\frac{1}{1-\sigma}\Big(\frac{h^\gamma}{c}\Big)^{\sigma-1},\quad c\in (0,R k].
	\end{equation*}
	Here, $a(k,h;c)= \left[\begin{smallmatrix}
		Bk-c \\ \rho(c-h)
	\end{smallmatrix}\right]$ and $\Sigma(k,h)=\left[\begin{smallmatrix}
		\beta_1k & 0 \\ 0 & \beta_2h
	\end{smallmatrix}\right]$. The maximum value Hamiltonian is
	\begin{align}\label{def_Hmax}
		\notag   &H_{\text{max}}(k,h,\pi,Q)=\sup_{c\in (0,Rk]}H(k,h,\pi,Q;c)
		\\&
		\qquad =
		Bk\pi_k-\rho h\pi_h+\frac{1}{2}\mathrm{tr}\big(\Sigma(k,h)^\top Q\Sigma(k,h)\big)
		+\sup_{c\in(0,R k]}\Big\{-c\pi_k+\rho c\pi_h-\frac{1}{\sigma-1}\Big(\frac{h^\gamma}{c}\Big)^{\sigma-1}\Big\}. 
	\end{align}
	Define by $g(c; k,h,\pi)=-c\pi_k+\rho c\pi_h-\frac{1}{\sigma-1}\big(\frac{h^\gamma}{c}\big)^{\sigma-1}$ for $c\in (0,Rk]$. From a  straightforward study of its derivative, we infer the following lemma.
	\begin{lemma}\label{gen_le_max_argmax}
		For every $k,h>0$ and $p=\left[\begin{smallmatrix} \pi_k\\\pi_h \end{smallmatrix}\right]\in\mathbb{R}^2$, the map $g(\cdot; k,h,\pi)$  attains its supremum on $(0,Rk]$ at a unique point. Specifically, defining by $G(k,h,\pi)=\sup_{c\in(0,R k]}g(c;k,h,\pi)$, 
		\begin{equation}\label{max}
			G(k,h,\pi)  =
			\begin{cases}
				-R k(\pi_k-\rho \pi_h)-\dfrac{1}{\sigma-1}\Big(\dfrac{h^\gamma}{R k}\Big)^{\sigma-1}, & \pi_k-\rho \pi_h
				\le  h^{\gamma(\sigma-1)}(R k)^{-\sigma},\\
				-\Big(1+\dfrac{1}{\sigma-1}\Big) [h^\gamma(\pi_k-\rho \pi_h)]^{1-\frac{1}{\sigma}},&\pi_k-\rho \pi_h
				>   {h^{\gamma(\sigma-1)}}{(R k)^{-\sigma}},
			\end{cases}
		\end{equation}
		and, defining by $\Gamma(k,h,\pi)=
		\underset{(0,Rk]}{\arg \max}\,\,g(\cdot;k,h,\pi)$, 
		\begin{equation}\label{arg_max}
			\Gamma(k,h,\pi)=\begin{cases}
				R k, & \pi_k-\rho \pi_h
				\le  h^{\gamma(\sigma-1)}(R k)^{-\sigma},\\
				\Big(\dfrac{h^{\gamma(\sigma-1)}}{\pi_k-\rho \pi_h}\Big)^{1/\sigma},&\pi_k-\rho \pi_h
				>   {h^{\gamma(\sigma-1)}}{(R k)^{-\sigma}}.
			\end{cases}
		\end{equation}
		In particular, $\Gamma$ is positive and has linear growth in $k$, namely, $0<\Gamma(k,h,\pi)\le Rk$.
	\end{lemma}
	
	The HJB equation associated with the control problem with unknown $v$ then reads, formally, 
	\begin{align}\label{eq:HJB-1}\notag
		& \theta v(k,h) = H_{\text{max}}	(k,h,\nabla v(k,h),D^2v(k,h))=
		\mathcal{L}_{k,h}v+G(k,h,\nabla v(k,h))
		,\\&\qquad \text{where }
		\mathcal{L}_{k,h}v\coloneqq 
		Bk\,\partial_kv(k,h)-\rho h\,\partial_hv(k,h)+\frac{1}{2}\mathrm{tr}\big(\Sigma(k,h)^\top D^2v(k,h)\Sigma(k,h)\big).
	\end{align}
	Here, $D^2v$ denotes  the Hessian matrix of  $v$.
	\vspace{1mm}
	
	Given a point $\pi\in\mathbb{R}^2_{++}$ and $r>0$, we denote by $B_r(\pi)$ the open ball in $\mathbb{R}^2$ centered at $\pi$ of radius $r$, and by $\overline{B}_r(\pi)$ its closure. For any $n\in\mathbb{N},$ let 
	$
	\mathcal{S}^n \text{ be the space of symmetric matrices in $\mathbb{R}^{n\times n}$.}
	$
	For $M,N\in\mathcal{S}^n$, the expression $M\le N$ means that $N-M\in\mathcal{S}^n$ is positive semidefinite. 
	We conclude this section with two lemmas  collecting some properties of  $H_{\text{max}}$ in \eqref{def_Hmax}; their proofs are postponed to Appendix \ref{Appendix}. In  Subsection \ref{sec_classical}, they  will be used to study the regularity of $V$ via PDE techniques. 
	
	\begin{lemma}\label{lemm:safonov}
		Let $D\subset \mathbb{R}^2$ be an open and bounded set such that its closure $\overline{D}\subset \mathbb{R}^2_{++}$.  Define
		\begin{equation}\label{def_Ftilde}
			\tilde{F}:\mathcal{S}^2\times\mathbb{R}^2\times\mathbb{R}\times \overline{D}\to\mathbb{R}\quad \text{by}\quad  \tilde{F}(Q, \pi, v, (k, h)) =-\theta v + H_{\emph{max}}(k,h,\pi,Q).
		\end{equation}
		Then there exists a  positive constant $\bar{C}>0$  such that, for every $Q=(Q_{ij})_{i,j}\in\mathcal{S}^2, \,(k,h)\in\overline{D},\, \pi=\left[\begin{smallmatrix} \pi_k\\\pi_h \end{smallmatrix}\right],  \bar{\pi}=\left[\begin{smallmatrix} \bar{\pi}_k\\\bar{\pi}_h  \end{smallmatrix}\right]\in\mathbb{R}^2$ and $v,\bar{v}\in\mathbb{R}$,
		\begin{align}\label{ineq:Saf2}
			|\tilde{F}(Q, \pi, v, (k, h)) - \tilde{F}(Q, \bar{\pi}, \bar{v}, (k, h))| \leq \bar{C}\big(|\pi_k - \bar{\pi}_k| + |\pi_h - \bar{\pi}_h| + |v - \bar{v}|\big),
		\end{align}
		and, for every $\alpha\in (0,1)$,
		\begin{equation} \label{eq_saf}
			\langle \tilde{F}(Q, \pi, v, \cdot) \rangle_{D}^{\alpha} \leq \bar{C}\bigg( 1+\sum_{i,j=1}^2 |Q_{ij}| + |\pi_k| + |\pi_h| + |v| \bigg),
		\end{equation} 
		where
		$
		\langle f\rangle^{\alpha}_{D} = \sup_{\substack{(k,h) \in D\\ \delta>0} } \delta^{-\alpha}\Big( \sup_{D \cap B_{\delta}(k,h)} f - \inf_{D \cap B_{\delta}(k,h)} f \Big), \text{ for a map  $f: D \to \mathbb{R}$.}$
	\end{lemma}
	
	\begin{lemma}\label{lemm:userguide}
		Let  $D\subset \mathbb{R}_{++}^2$ be an  open and bounded set such that its closure $\overline{D}\subset \mathbb{R}^2_{++}$.   Define the function $\tilde{F}:\mathcal{S}^2\times\mathbb{R}^2\times\mathbb{R}\times D\to\mathbb{R}$ as in Lemma \ref{lemm:safonov}. Then there is a constant $\widetilde{C}>0$ such that
		\begin{align}\label{est_user}
			\notag   & \tilde{F}(Q, \alpha(k-\bar{k}, h-\bar{h}), v,(k, h)) - \tilde{F}(\bar{Q}, \alpha(k-\bar{k}, h-\bar{h}),v, (\bar{k}, \bar{h})) \\&\qquad 
			\leq \widetilde{C}(\alpha  |(k-\bar{k}, h-\bar{h})|^2 + |(k-\bar{k}, h-\bar{h})|),
		\end{align}
		for every $\alpha>0$,  $(k,h),\,(\bar{k},\bar{h})\in D,\,v\in\mathbb{R}$ and $Q,\bar{Q}\in\mathcal{S}^2$ satisfying
		\begin{equation}\label{ineq:QIQ}
			\begin{bmatrix}
				Q && 0 \\
				0 && -\bar{Q} 
			\end{bmatrix} \leq 3\alpha \begin{bmatrix}
				I && -I \\
				-I && I 
			\end{bmatrix}.
		\end{equation}
	\end{lemma}
	\noindent

	\section{Properties of the Value Function $V$}\label{sec_regularity}
	
	This section analyzes  the value function $V$ in \eqref{def:valuefunction}. In particular, we will show  that $V$ is finite-valued and continuous, from which we deduce the dynamic programming principle in Proposition \ref{DPP}. 
	We start with a  preliminary result concerning the homogeneity of $V$ in $\mathbb{R}^2_{++}$.
	\begin{lemma} \label{lemma_hom} 
		For every $\alpha>0$,
		\begin{equation}\label{eq_homo}
			V(\alpha k_0,\alpha h_0)=\alpha^{(1-\gamma)(1-\sigma)}V(k_0,h_0),\quad (k_0,h_0)\in\mathbb{R}^2_{++}.
		\end{equation}
	\end{lemma}
	\begin{proof}
		Fix $\alpha>0$ and an initial condition $(k_0,h_0)\in\mathbb{R}^2_{++}.$ Recall the definition of $\mathcal{A}(k_0, h_0)$ and $\mathcal{A}(\alpha k_0,\alpha h_0)$  in \eqref{def:setofcontrols}. According to  Lemma \ref{set_controls}, these sets only depend on $k_0$ and $\alpha$, i.e., they are independent of $h_0$. Moreover, by the explicit expression of the capital $k$ in \eqref{expl_eq_k}, it is straightforward to conclude that  
		$    \mathcal{A}(\alpha k_0,\,\alpha h_0)=\alpha \mathcal{A}(k_0, h_0).$
		Using the notation introduced in \eqref{notation_dependence}, by  \eqref{expl_eq_h}
		\[
		h^{\alpha h_0,\alpha c}_t=\alpha h^{ h_0, c}_t,\quad t\ge0,\mathbb{P}-\text{a.s., for every }c\in \mathcal A( k_0, h_0).
		\] 
		As a result, also using \eqref{def_J} and \eqref{def:valuefunction}, 
		\begin{align*}
			V(\alpha k_0,\alpha h_0) 
			&= \sup_{ c\in\mathcal{A}(k_0,h_0)}\mathbb{E}\bigg[\int_{0}^\infty \frac{1}{1-\sigma}\bigg(\frac{{(h^{\alpha h_0,\alpha c}_t)}^\gamma}{ \alpha c_t}\bigg)^{\sigma-1}e^{-\theta t}\de t\bigg]\\
			&= \sup_{ c\in\mathcal{A}(k_0,h_0)}\mathbb{E}\bigg[\int_{0}^\infty \frac{1}{1-\sigma}\bigg(\frac{{(\alpha h^{h_0,c}_t)}^\gamma}{ \alpha c_t}\bigg)^{\sigma-1}e^{-\theta t}\de t\bigg] =\alpha^{(1-\gamma)(1-\sigma)}V(k_0,h_0).
		\end{align*}
		The proof is then complete.
	\end{proof}

	Since the objective functional $J$ in \eqref{def_J} is nonpositive, by the definition \eqref{def:valuefunction}, $V\le 0$ in $\mathbb{R}^2_{++}$, as well. Recalling our standing Assumption \ref{assumption1}, the following proposition gives a lower bound for $V$. In particular, we deduce that $V(k_0,h_0)\in (-\infty,0]$ (i.e., $V(k_0,h_0)$ is finite) for every $(k_0,h_0)\in\mathbb{R}^2_{++}.$
	\begin{proposition}\label{prop:lbound}
		There exists a positive constant $C=C(\sigma,\gamma,\theta, \rho,\beta_1,\beta_2,B)>0$ such that 
		\begin{equation}\label{eq_lower}
			V(k_0,h_0)\geq -C\bigg(\left(\frac{h_0^{\gamma}}{k_0}\right)^{\sigma-1}+\frac{1}{k_0^{(1-\gamma)(\sigma-1)}}\bigg),\quad (k_0,h_0)\in\mathbb{R}^2_{++}.
		\end{equation}
	\end{proposition}
	\begin{proof}
		Fix $(k_0,h_0)\in\mathbb{R}^2_{++}$ and consider the control $c_t=Bk_0\mathcal{E}(\beta_1W_{1,t}),\,t\ge 0$. By \eqref{expl_eq_k},
		\begin{equation*}
			k_t^{k_0,c} =:k_t = k_0\mathcal{E}(\beta_1W_{1,t})=\frac{c_t}{B}>0,\quad t\ge 0.
		\end{equation*}
		Since $B\le R$ by Assumption \ref{assumption1}, the constraint \eqref{bdd_control} is fulfilled. Thus, also recalling  Lemma \ref{set_controls}, we have $c\in\mathcal{A}(k_0,h_0)$. Using the explicit expressions of the habit process $h$ in \eqref{expl_eq_h} and the Doléans-Dade exponential, by \eqref{def:valuefunction},
		\begin{align*}
			V(k_0,h_0)&\geq -\frac{(Bk_0)^{1-\sigma}}{\sigma-1}\mathbb{E} \bigg[\int_0^\infty e^{-\theta t} e^{(1-\sigma)\big[(\gamma\rho+\frac{1}{2}(\gamma\beta_2^2-\beta_1^2))t+\beta_1W_{1,t}-\gamma\beta_2 W_{2,t}\big]}\\
			& \hspace{3cm}\times\Big(h_0+\rho Bk_0\int_0^te^{(\rho+\frac{1}{2}(\beta_2^2-\beta_1^2))s+\beta_1 W_{1,s}-\beta_2 W_{2,s}}\de s \Big)^{\gamma(\sigma-1)}\de t\bigg]\\
			& \eqqcolon-\frac{(Bk_0)^{1-\sigma}}{\sigma-1}\mathbb{E}\bigg[\int_0^\infty \mathbf{\upperRomannumeral{1}}_t\times \mathbf{\upperRomannumeral{2}}_t\,\de t\bigg].
		\end{align*}
		By Tonelli's theorem  and H\"older's inequality, it then follows that  
		\begin{equation}\label{bound_V_inter1}
			V(k_0,h_0)\ge -\frac{(Bk_0)^{1-\sigma}}{\sigma-1} \int_0^\infty \big(\mathbb{E}\big[|\mathbf{\upperRomannumeral{1}}_t|^p\big]\big)^{\frac{1}{p}} \big(\mathbb{E}\big[|\mathbf{\upperRomannumeral{2}}_t|^q\big]\big)^{\frac1q}\de t   ,
		\end{equation}
		where we pick $q>2$ such that $q\gamma(\sigma-1)>1$, and $p={q}{(q-1)^{-1}}$. In particular,  $q>2$ entails $p<q$.  
		Since $W_1$ and $W_2$ are independent Brownian motions and, for every $t>0$, the random variables $e^{\beta_1W_{1,t}}$ and $e^{-\gamma\beta_2W_{2,t}}$ have a $\log$-normal distribution,	
		\begin{equation*}
			\mathbb{E}[|\mathbf{\upperRomannumeral{1}}_t|^p]= \exp\left\{p(1-\sigma)\left(\gamma\rho +\frac{1}{2}(\gamma \beta_2^2-\beta_1^2)+\frac{\theta}{\sigma-1}+\frac{1}{2}p(1-\sigma)(\beta_1^2+\gamma^2\beta_2^2)\right)t\right\},\quad t\ge0.
		\end{equation*}
		As for $\mathbf{\upperRomannumeral{2}}$, defining $q^*=q\gamma(\sigma-1)>1$ and denoting by $p^*$ its H\"older conjugate, we compute
		\begin{align*}
			\mathbb{E}[|\mathbf{\upperRomannumeral{2}}_t|^q]
			&\le 2^{q^*-1}\bigg( h_0^{q^*}+(\rho Bk_0)^{q^*}\bigg(\int_0^te^{p^*(\rho+\frac{1}2(\beta_2^2-\beta_1^2))s}\de s\bigg)^{\frac{q^*}{p^*}}\mathbb{E}\bigg[\int_0^te^{q^*(\beta_1 W_{1,s}-\beta_2 W_{2,s})}\de s \bigg]\bigg)\\
			&= 2^{q^*-1}\bigg( h_0^{q^*}+(\rho Bk_0)^{q^*}\bigg(\int_0^te^{p^*(\rho+\frac{1}2(\beta_2^2-\beta_1^2))s}\de s\bigg)^{\frac{q^*}{p^*}}\int_0^te^{\frac12(q^*)^2(\beta_1^2 +\beta_2^2)s}\de s \bigg)\\
			& \le 2^{q^*-1}\bigg( h_0^{q^*}+C_1k_0^{q^*} \bigg(\int_0^te^{p^*(\rho+\frac{1}2(\beta_2^2-\beta_1^2))s}\de s\bigg)^{\frac{q^*}{p^*}} \exp\left\{\frac12(q^*)^2(\beta_1^2 +\beta_2^2) t\right\}\bigg),\quad t\ge0,
		\end{align*}
		for some constant $C_1=C_1(\rho,\beta_1,\beta_2,q^\ast, B)>0$.
		Here, in the first estimate we apply H\"older's inequality with exponents $p^*$ and $q^*$. Plugging the two previous estimates on $\mathbf{\upperRomannumeral{1}}$ and $\mathbf{\upperRomannumeral{2}}$ into \eqref{bound_V_inter1} and recalling that $q^*/q=\gamma(\sigma-1),$ using also that the power $x^{\frac{1}{q}}$ is subadditive on $[0,\infty),$ we obtain
		\begin{align}\label{recall_cases}
			\notag
			& V(k_0,h_0)\ge -C_2\,k_0^{1-\sigma} \int_0^\infty  \exp\left\{(1-\sigma)\left(\gamma\rho +\frac{1}{2}(\gamma \beta_2^2-\beta_1^2)+\frac{\theta}{\sigma-1}+\frac{1}{2}p(1-\sigma)(\beta_1^2+\gamma^2\beta_2^2)\right)t\right\}\\
			&\quad\times \bigg( h_0^{\gamma(\sigma-1)}+k_0^{\gamma(\sigma-1)}  \bigg(\int_0^te^{p^*(\rho+\frac{1}2(\beta_2^2-\beta_1^2))s}\de s\bigg)^{\frac{\gamma(\sigma-1)}{p^*}}
			\exp\left\{ \frac12q\gamma^2(\sigma-1)^2(\beta_1^2 +\beta_2^2) t\right\}\bigg)\de t,
		\end{align}
		for a constant $C_2=C_2(\sigma,q, \rho,\beta_1,\beta_2,q^\ast, B)>0.$ 
		\vspace{2mm}
		
		From now on, we denote by $C_i,\,i=3,4,\dots,$ positive constants possibly depending on the parameters $(\sigma,q,\rho,\beta_1,\beta_2,p^\ast,q^\ast , B)$.
		We now argue that the integral on the right-hand side of \eqref{recall_cases} converges thanks to Assumption \ref{assumption1}. To do this, we distinguish two cases. \\
		If $\rho+\frac{1}{2}(\beta_2^2-\beta_1^2)\le 0$, then
		\[
		\int_0^te^{p^*(\rho+\frac{1}2(\beta_2^2-\beta_1^2))s}\de s\le 1+t,\quad t\ge 0.
		\]
		Hence, recalling that $q>p$ and $\gamma\in(0,1]$, we deduce from \eqref{recall_cases} that
		\begin{align*}
			&V(k_0,h_0)\ge -C_3\bigg(\left(\frac{h_0^{\gamma}}{k_0}\right)^{\sigma-1}+\frac{1}{k_0^{(1-\gamma)(\sigma-1)}}\bigg) \int_0^\infty \Big(1+t^{\frac{\gamma(\sigma-1)}{p^*}}\Big)
			\exp\left\{(\sigma-1)\,e_1(q) \,t\right\}\de t, \\
			&\qquad \text{where } e_1(q)\coloneqq      -\gamma\rho-\frac{1}{2}(\gamma\beta_2^2-\beta_1^2)-\frac{\theta}{\sigma-1}+   q(\sigma-1)(\beta_1^2 +\gamma^2\beta_2^2) .
		\end{align*}
		Condition \eqref{eq_maggiore-minore}   enables us to find a $q>\max\left\{2,\frac{1}{\gamma(\sigma-1)}\right\}$ such that $e_1(q)<0$,
		whence \eqref{eq_lower}.\\
		In the case $\rho+\frac{1}{2}(\beta_2^2-\beta_1^2)> 0$, using again the fact that $q>p$ and $\gamma\in(0,1]$, from \eqref{recall_cases} we obtain the estimate
		\begin{align*}
			&V(k_0,h_0)\ge  -C_4\bigg(\left(\frac{h_0^{\gamma}}{k_0}\right)^{\sigma-1}+\frac{1}{k_0^{(1-\gamma)(\sigma-1)}}\bigg) \int_0^\infty \exp\left\{(\sigma-1)\,e_2(q)  \,t\right\}\de t,    \\
			&\qquad \text{where } e_2(q)\coloneqq      \frac{1}2\beta_1^2(1-\gamma)-\frac{\theta}{\sigma-1}+   q(\sigma-1)(\beta_1^2 +\gamma^2\beta_2^2)  .
		\end{align*}
		As before, Condition \eqref{eq_maggiore-minore} in Assumption \ref{assumption1} enables us to find a $q>\max\left\{2,\frac{1}{\gamma(\sigma-1)}\right\}$ such that $e_2(q)<0$,
		whence \eqref{eq_lower}. 
		The proof is now complete. 
	\end{proof}
	
	The following proposition shows the monotonicity of the value function $V$ in $\mathbb{R}^2_{++}$ with respect to its variables. 
	\begin{proposition}\label{prop:Vmonotone}
		The value function $V$ is nondecreasing in $k$ and nonincreasing in $h$. 
	\end{proposition}
	\begin{proof}
		We start from the monotonicity of $V$ in $k$. Fix $h_0\in (0,\infty) $ and $k_1,\,k_2\in (0,\infty)$ such that $k_1<k_2$. By the explicit expression of the capital process in \eqref{expl_eq_k},  $ \mathcal{A}(k_1,h_0)\subset \mathcal{A}(k_2,h_0)$. Since the objective functional $J$ in \eqref{def_J} does not depend on the capital $k$, we deduce that $V(k_1,h_0)\le V(k_2,h_0)$.\\
		As for the monotonicity in $h$, fix $k_0\in (0,\infty)$ and $h_1,\,h_2\in (0,\infty)$ such that $h_1<h_2$. By Lemma \ref{set_controls}, $\mathcal{A}(k_0,h_1)=\mathcal{A}(k_0,h_2)=\mathcal{A}(k_0)$. Moreover, from the explicit expression in \eqref{expl_eq_h}, we have $h^{h_1,c}\le h^{h_2,c}$, $\mathbb{P}\otimes \de t-$a.e., for every $c\in\mathcal{A}(k_0)$. As a result,   
		$
		J(k_0,h_1;c)\ge J(k_0,h_2;c),\, c\in \mathcal{A}(k_0).
		$
		Taking the supremum over $\mathcal{A}(k_0)$ yields $V(k_0,h_1)\ge V(k_0,h_2)$, which completes the proof.
	\end{proof}
	Combining the homogeneity in Lemma \ref{lemma_hom} and the monotonicity in Proposition \ref{prop:Vmonotone},  
	we now demonstrate the continuity of $V$ in $\mathbb{R}^2_{++}$.
	\begin{theorem}	\label{prop_cont}
		The  value function $V\colon \mathbb{R}^2_{++}\to[-\infty,0]$ is continuous.
	\end{theorem}
	\begin{proof}
		From Proposition \ref{prop:lbound},  $V\neq -\infty$. Fix an arbitrary point $P_0= (k_0,h_0)\in\mathbb{R}_{++}^2$.  Note  that, by Lemma \ref{lemma_hom}, the separate continuity of $V(k_0,\cdot)$ at $h_0$, i.e., 
		\begin{equation}\label{cont_h}
			\lim_{h\to h_0}V(k_0,h)=V(k_0,h_0),
		\end{equation}
		is sufficient to deduce the joint continuity of $V$ at $P_0$. Indeed, given two sequences $(h_n)_n,(k_n)_n\subset \mathbb{R}_{++}$ which converge to $h_0$ and $k_0$, respectively, by \eqref{eq_homo} and \eqref{cont_h},
		\[
		V(k_n,h_n)= V\left(\frac{k_n}{k_0}k_0,\frac{k_n}{k_0} \frac{k_0}{k_n}h_n\right)=\left(\frac{k_n}{k_0}\right)^{(1-\gamma)(1-\sigma)}V\left(k_0, \frac{k_0}{k_n}h_n\right)\underset{n\to\infty}{\longrightarrow}V(k_0,h_0).
		\]
		We then focus on \eqref{cont_h}. Since $V$ is nonincreasing with respect to  $h$ by Proposition \ref{prop:Vmonotone}, 
		\begin{equation*}
			\lim_{h\to h_0^-}V(k_0,h)\geq V(k_0,h_0) \geq \lim_{h\to h_0^+}V(k_0,h).
		\end{equation*}
		Thus, we only need to show the reversed chain of inequalities, starting from \begin{equation}\label{simple_h1}\lim_{h\to h_0^+}V(k_0,h)\geq V(k_0,h_0).
		\end{equation}
		Consider $(h_n)_n\subset (h_0,\infty)$ such that $\lim_{n\to\infty}h_n=h_0$. For every $\epsilon>0$, by \eqref{def:valuefunction}, there exists an $\epsilon-$optimal control $\bar{c}\in \mathcal{A}(k_0,h_0)$ satisfying
		\begin{equation*}
			V(k_0,h_0)\leq \mathbb{E} \bigg[\int_{0}^{\infty}\frac{1}{1-\sigma}e^{-\theta t}\frac{1}{\bar{c}_{t}^{\sigma-1}}\big({h}_{t}^{h_0,\bar{c}}\big)^{\gamma(\sigma-1)}\de t\bigg] + \epsilon.
		\end{equation*}
		From Proposition \ref{prop:lbound} the value function is finite. It then follows that 
		\begin{equation}\label{adj_nofaou}
			e^{-\theta t}\frac{1}{\bar{c}_{t}^{\sigma-1}}\big({h}_{t}^{h_0,\bar{c}}\big)^{\gamma(\sigma-1)}\in L^1(\Omega \times \mathbb{R}_+).
		\end{equation}
		Since $\bar{c}\in\mathcal{A}(k_0,h_n)=\mathcal{A}(k_0)$ by Lemma \ref{set_controls}, we have 
		\begin{align*}
			&  V(k_0,h_n)-V(k_0,h_0) \geq \frac{1}{1-\sigma}\mathbb{E}\bigg[ \int_{0}^{\infty}\frac{e^{-\theta t}}{\bar{c}_{t}^{\sigma-1}}\bigg(\Big({h}^{h_n,\bar{c}}_{t}\Big)^{\gamma(\sigma-1)}-\Big({h}^{h_0,\bar{c}}_{t}\Big)^{\gamma(\sigma-1)}\bigg)\de t\bigg] -\epsilon,
		\end{align*}
		which holds for every $n\in \mathbb{N}$. The explicit expression in \eqref{expl_eq_h} implies that ${h}^{h_n,\bar{c}}$ converges  to ${h}^{h_0,\bar{c}}$ locally uniformly on $[0,\infty)$, $\mathbb{P}-$a.s., and there exists a constant $C>1$ such that 
		$  
		0\le {h}^{h_n,\bar{c}}\le C  {h}^{h_0,\bar{c}},\, n\in\mathbb{N}, \,\mathbb{P}\otimes \de t-\text{a.e.}
		$
		Therefore, by \eqref{adj_nofaou}, the dominated convergence theorem yields 
		\(
		\lim_{n\to\infty}  V(k_0,h_n)\geq V(k_0,h_0)-\epsilon.
		\)
		This estimate gives \eqref{simple_h1} as $\epsilon$ is arbitrary.
		\vspace{2mm}\\
		Concerning $\lim_{h\to h_0^-}V(k_0,h)$, we suppose by contradiction that 
		\begin{equation}\label{assurdo}
			\lim_{h\to h_0^-}V(k_0,h)>V(k_0,h_0).
		\end{equation}
		By \eqref{def:valuefunction}, we consider a sequence $(h_n)_n\subset (0,h_0)$ which converges to $h_0$ and corresponding  admissible controls $c_n\in\mathcal{A}(k_0,h_n)$  such that 
		\begin{equation}\label{bound_contradiction}
			J(k_0,h_n;c_n)>V(k_0,h_n)-\frac{1}{n},\, n\in\mathbb{N}, \text{\quad whence, by (\ref{assurdo}), }\liminf_{n\to\infty}J(k_0,h_n;c_n)>V(k_0,h_0).
		\end{equation}
		Note that  \eqref{bound_contradiction} combined with \eqref{expl_eq_h} and \eqref{def_J} entails that 
		\begin{equation}\label{unif_bound}
			\sup_{n\in\mathbb{N}}\mathbb{E}\bigg[\int_0^\infty \frac{e^{-\theta t}}{c_{n,t}^{\sigma-1}}  \mathcal{E}(-\rho t+\beta_2 W_{2,t})^{\gamma(\sigma-1)}\bigg(1+\int_0^t\frac{c_{n,s}}{\mathcal{E}(-\rho s+\beta_2W_{2,s})}\,\de s\bigg)^{\gamma(\sigma-1)}\de t\bigg]<\infty.
		\end{equation}
		Fix a generic $n\in\mathbb{N}$. Since $h_n<h_0$, the modified control $\frac{h_n}{h_0}c_n$ belongs to  $\mathcal{A}(k_0,h_n).$ We then consider $J\Big(k_0,h_n;\frac{h_n}{h_0}c_n\Big)$    and aim to estimate the difference 
		$
		\Delta_n\coloneqq J\Big(k_0,h_n;\frac{h_n}{h_0}c_n\Big)-J(k_0,h_n;c_n).
		$ 
		In particular, by  \eqref{def_J}  we compute 
		\begin{align}\notag
			|\Delta_n|&\le \frac{1}{\sigma-1}\bigg(\left(\frac{h_0}{h_n}\right)^{\sigma-1}-1\bigg) \mathbb{E}\bigg[ \int_0^\infty \frac{e^{-\theta t}}{c_{n,t}^{\sigma-1}} \Big( h_{t}^{h_n,\frac{h_n}{h_0}c_n}    \Big)^{\gamma(\sigma-1)}\de t    \bigg] \\  \label{est_delta}
			&\quad +\frac{1}{\sigma-1} \mathbb{E}\bigg[ \int_0^{\infty} \frac{e^{-\theta t}}{c_{n,t}^{\sigma-1}} \Big|\Big(   h_{t}^{h_n,\frac{h_n}{h_0}c_n} \Big)^{\gamma(\sigma-1)}- \Big(
			h_{t}^{h_n,c_n} \Big)^{\gamma(\sigma-1)}\Big| \de t  \bigg]\eqqcolon\mathbf{\upperRomannumeral{1}}_n+\mathbf{\upperRomannumeral{2}}_n.
		\end{align}
		By  \eqref{expl_eq_h},
		\begin{equation}\label{rel_h0}
			h_{t}^{h_n,\frac{h_n}{h_0}c_n}  = \frac{h_n}{h_0}h_{t}^{h_0,c_n},\quad t\ge 0,
		\end{equation}
		whence, considering also that the sequence $(J(k_0,h_0;c_n))_n$ is bounded by \eqref{unif_bound},
		\[
		\mathbf{\upperRomannumeral{1}}_n \le   \bigg(\left(\frac{h_0}{h_n}\right)^{\sigma-1}-1\bigg)\bigg(\frac{h_n}{h_0}\bigg)^{\gamma(\sigma-1)} \sup_{n\in\mathbb{N}}|J(k_0,h_0;c_n)|<
		\frac{1}{2n},\quad  \text{for $n$ sufficiently large.}\footnote{We assume without loss of generality that the upper bound is $\frac{1}{2n}$ for notational convenience, but we could replace it with any infinitesimal sequence. The same applies to the following upper bounds for $\mathbf{\upperRomannumeral{2}}_n$.}
		\]
		As for $\mathbf{\upperRomannumeral{2}}_n$, we distinguish two cases: $\gamma(\sigma-1)\le 1$ and $\gamma(\sigma-1)> 1$. In the former, i.e., $\gamma(\sigma-1)\le 1$, we notice that the power function $x^{\gamma(\sigma-1)}$ is $\gamma(\sigma-1)-$H\"older continuous on $[0,\infty)$. Denoting by $C_1$ its H\"older-continuity constant,  by \eqref{expl_eq_h} and \eqref{unif_bound} we deduce that 
		\begin{align*}
			\mathbf{\upperRomannumeral{2}}_n&\le   \frac{C_1}{\sigma-1} \left(1-\frac{h_n}{h_0}\right)^{\gamma(\sigma-1)} \\
			&\quad\times   \mathbb{E}\bigg[   \int_0^{\infty} \frac{e^{-\theta t}}{c_{n,t}^{\sigma-1}} \mathcal{E}(-\rho t+\beta_2 W_{2,t})^{\gamma(\sigma-1)} \left(\rho\int_0^{t} \frac{c_{n,s}}{\mathcal{E}(-\rho s+\beta_2W_{2,s})}\,\de s\right)^{\gamma(\sigma-1)} \de t  \bigg]\\
			&\le C_1\left(1-\frac{h_n}{h_0}\right)^{\gamma(\sigma-1)} \sup_{n\in\mathbb{N}}|J(k_0,h_0;c_n)|\le \frac{1}{2n},\quad \text{for $n$ sufficiently large}.
		\end{align*}
		In the latter case $\gamma(\sigma-1)>1$, recalling  that $h_n<h_0$, the mean value theorem applied to the map $x^{\gamma(\sigma-1)}$ yields the existence of a stochastic process $\xi_n\colon \Omega\times \mathbb{R}_+\to \mathbb{R}$ such that, for every $(\omega,t)\in \Omega \times \mathbb{R}_+$, 
		$
		\xi_n\in \big(h^{h_n,\frac{h_n}{h_0}c_n},h^{h_n,c_n}\big)
		$
		and 
		\begin{align*}
			\Big|  \Big(   h^{h_n,\frac{h_n}{h_0}c_n} \Big)^{\gamma(\sigma-1)}-  \Big(h^{h_n,c_n}\Big)^{\gamma(\sigma-1)}\Big| &=\gamma(\sigma-1) \xi_n^{\gamma(\sigma-1)-1} \Big(h^{h_n,c_n}-h^{h_n,\frac{h_n}{h_0}c_n}\Big) \\
			& \le  \gamma(\sigma-1)   \big(h^{h_n,c_n}\big)^{\gamma(\sigma-1)-1}  \Big(h^{h_n,c_n}-h^{h_n,\frac{h_n}{h_0}c_n}\Big),
		\end{align*}
		where in the final step we use that $x^{\gamma(\sigma-1)-1}$ is increasing on $\mathbb{R}_+$.
		Thanks to this inequality and \eqref{expl_eq_h}, since $(x+y)^{\gamma(\sigma-1)-1}\le \max\{1,2^{\gamma(\sigma-1)-2}\} (x^{\gamma(\sigma-1)-1}+y^{\gamma(\sigma-1)-1})$ for every $x,y\ge 0,$
		\begin{equation*}
			\mathbf{\upperRomannumeral{2}}_n \le \gamma \max\{1,2^{\gamma(\sigma-1)-2}\} \bigg(1-\frac{h_n}{h_0}\bigg) \mathbb{E}\bigg[ \int_0^{\infty} \frac{e^{-\theta t}}{c_{n,t}^{\sigma-1}} \mathcal{E}(-\rho t+\beta_2 W_{2,t})^{\gamma(\sigma-1)} \big(\mathbf{\upperRomannumeral{2}}_{n,1,t}+\mathbf{\upperRomannumeral{2}}_{n,2,t}\big)\de t 
			\bigg],
		\end{equation*}
		with
		\begin{equation*}
			\begin{aligned}
				&\mathbf{\upperRomannumeral{2}}_{n,1,t}\coloneqq  \rho h_n^{\gamma(\sigma-1)-1} \bigg(\int_0^{t} \frac{c_{n,s}}{\mathcal{E}(-\rho s+\beta_2W_{2,s})}\,\de s\bigg),\quad
				\mathbf{\upperRomannumeral{2}}_{n,2,t}\coloneqq  \bigg(\rho\int_0^{t} \frac{c_{n,s}} {\mathcal{E}(-\rho s+\beta_2W_{2,s})}\,\de s\bigg)^{\gamma(\sigma-1)}.
			\end{aligned}
		\end{equation*}
		Observing that 
		$   
		\mathbf{\upperRomannumeral{2}}_{n,1,t}        \le        h_0^{\gamma(\sigma-1)-1}\big(1 +\mathbf{\upperRomannumeral{2}}_{n,2,t}        \big),
		$
		by \eqref{unif_bound} we obtain
		\begin{align*}
			\mathbf{\upperRomannumeral{2}}_n&\le       \gamma \max\{1,2^{\gamma(\sigma-1)-2}\}       \max\big\{1,h_0^{\gamma(\sigma-1)-1}\big\}\bigg(1-\frac{h_n}{h_0}\bigg)\\
			&\quad \times  \mathbb{E}\bigg[   \int_0^{\infty} \frac{e^{-\theta t}}{c_{n,t}^{\sigma-1}} \mathcal{E}(-\rho t+\beta_2 W_{2,t})^{\gamma(\sigma-1)}  \bigg(1+2\bigg(\rho\int_0^{t} \frac{c_{n,s}}    {\mathcal{E}(-\rho s+\beta_2W_{2,s})}\,\de s\bigg)^{\gamma(\sigma-1)}\,\bigg) \de t \bigg]\\
			&< \frac{1}{2n},\quad \text{for $n$ sufficiently large}.
		\end{align*}
		
		Plugging the estimates we have just established on $\mathbf{\upperRomannumeral{1}}_n$ and $\mathbf{\upperRomannumeral{2}}_n$ into \eqref{est_delta} gives $|\Delta_n|<\frac{1}{n}$, for $n$ sufficiently large. Thus, by \eqref{bound_contradiction},
		\[
		\liminf_{n\to\infty} J\Big(k_0,h_n;\frac{h_n}{h_0}c_n\Big)=  \liminf_{n\to\infty}J(k_0,h_n;c_n)+\lim_{n\to\infty}\Delta_n=\liminf_{n\to\infty}J(k_0,h_n;c_n)>V(k_0,h_0).
		\]
		At the same time, recalling \eqref{def_J} and \eqref{rel_h0}, 
		\[
		J\Big(k_0,h_n;\frac{h_n}{h_0}c_n\Big)   =\left(\frac{h_0}{h_n}\right)^{(1-\gamma)(\sigma-1)} J(k_0,h_0;c_n)\le \left(\frac{h_0}{h_n}\right)^{(1-\gamma)(\sigma-1)}V(k_0,h_0),
		\]
		whence $\liminf_{n\to\infty}J\Big(k_0,h_n;\frac{h_n}{h_0}c_n\Big)\le V(k_0,h_0)$.
		This is a contradiction demonstrating that \eqref{assurdo} cannot hold. Consequently, $V(k_0,\cdot)$ is also left-continuous at $h_0$ and the proof is complete. 
	\end{proof}
	Proposition \ref{prop:lbound} and Theorem \ref{prop_cont} enable us to establish the dynamic programming principle, as  shown in the next result. Its proof,  which follows standard arguments (see, e.g., \cite{fleming2006controlled,gassiat2014investment,ishii2002class}), is omitted for brevity.
	\begin{proposition}\label{DPP} 
		The value function $V$ satisfies the dynamic programming equation, that is, for every $(k_0,h_0)\in\mathbb{R}_{++}^2 $ and almost surely finite stopping time $\tau$ (possibly depending on $c \in \mathcal{A}(k_0,h_0)$),  the following functional equation  holds:
		\begin{equation}
			V(k_0,h_0)=\sup _{c \in \mathcal{A}(k_0,h_0)} \mathbb{E}\bigg[\int_{0}^{\tau}\frac{e^{-\theta t}}{1-\sigma}\bigg(\frac{(h_t^{h_0,c})^{\gamma}}{c_t}\bigg)^{\sigma-1}\de t + e^{-\theta \tau}V(k^{k_0,c}_\tau,h^{h_0,c}_\tau) \bigg].
		\end{equation}
	\end{proposition}

	\section{Studying the HJB Equation} \label{sec_new}

	In this section, we show that the value function $V$ solves the HJB equation \eqref{eq:HJB-1} both in the viscosity and the classical sense. In particular, we demonstrate that  $V$ is a regular map in $C^2(\mathbb{R}^2_{++};\mathbb{R}).$
	
	\subsection{$V$ is a Viscosity Solution of the HJB Equation} \label{sec_newvisc}
	
	We begin this subsection by recalling the definition of (continuous) \emph{viscosity solutions} of the HJB equation \eqref{eq:HJB-1}.
	
	\begin{definition}\label{def_viscosity}
		Let $D$ be an open set in $\mathbb{R}_{++}^2$. A continuous function $v:D \rightarrow \mathbb{R}$ is a \emph{viscosity subsolution} (resp., \emph{viscosity supersolution}) of \eqref{eq:HJB-1} in  $D$ if, for any $\psi \in C^2(D;\mathbb{R})$ and any local maximum point $(k_M,h_M) \in D$ (resp., local minimum point $(k_m,h_m) \in D$) of $v-\psi$, the following inequality holds:
		\begin{align*}	
			\theta v(k_M,h_M)- H_{\text{max}}(k_M,h_M,\nabla \psi(k_M,h_M), D^2\psi(k_M,h_M)) &\leq 0\\
			(\text{resp., } \theta v(k_m,h_m)-  H_{\text{max}}(k_m,h_m,\nabla \psi(k_m,h_m), D^2\psi(k_m,h_m))  &\geq 0).
		\end{align*}
		A continuous function $v:D\rightarrow \mathbb{R}$ is a \emph{viscosity solution} of \eqref{eq:HJB-1} in  $D$ if it is both a viscosity subsolution and a viscosity supersolution in $D$.
	\end{definition}
	In the next theorem, adapting in a nontrivial way some standard arguments in, e.g., \cite{choulli2003diffusion,di2011pension}, we prove that the value function $V$ defined in \eqref{def:valuefunction} solves \eqref{eq:HJB-1} on $\mathbb{R}^2_{++}$ in the viscosity sense.

	\begin{theorem}\label{prop:VisVisc}
		The value function $V$ is a viscosity solution of the HJB equation \eqref{eq:HJB-1} in $\mathbb{R}^2_{++}$.
	\end{theorem}
	\begin{proof}
		We  structure the proof in two steps. In the first (resp., second) step, we demonstrate that $V$ is a viscosity supersolution (resp., viscosity subsolution) of \eqref{eq:HJB-1} on $\mathbb{R}^2_{++}$.\vspace{1mm}\\
		\underline{\emph{Step \upperRomannumeral{1}}}.  Consider $\psi \in C^2(\mathbb{R}_{++}^2; \mathbb{R} )$ and a local minimum point $(k_m,h_m) \in\mathbb{R}_{++}^2$ for $V-\psi$. Without loss of generality, we suppose that $V(k_m,h_m)=\psi(k_m,h_m)$, hence there exists $\delta>0$ such that $V(k,h)\geq \psi(k,h)$, $(k,h)\in B_\delta (k_m,h_m)\subset \mathbb{R}_{++}^2$.   
		Fix a constant $\bar{c}\in (0,Rk_m)$ and  define   the stopping time
		\[
		\bar{\tau}=\inf\left\{ t\ge 0 : k^{k_m,\bar{c}}_t\le \frac{\bar{c}}{R}\right\};
		\]
		observe that $\bar{\tau}>0,\,\mathbb{P}-$a.s., by the continuity of $k^{k_m,\bar{c}}$.  We now employ $\bar{\tau}$ to construct the following positive progressively measurable process  $c_{\nu}^\ast,\,\nu>0$: 
		\begin{equation}\label{adm_control_cast}
			c^*_{\nu, t} = \begin{cases}
				\bar{c}, & t\le\bar{\tau},\\
				\nu    \big(\mathcal{E}(B\bar{\tau}+\beta_1W_{1,\bar{\tau}})\big)^{-1} \mathcal{E}(\beta_1W_{1,t}),& t > \bar{\tau}.
			\end{cases} 
		\end{equation}
		Notice that, choosing $\nu$ sufficiently small, $c_\nu^*\in \mathcal{A}(k_m,h_m)$.  Indeed, by the definition of $\bar{\tau}$ and \eqref{expl_eq_k}, on the event $\{\bar{\tau}<\infty\}$ 
		\[
		\frac{\bar{c}}{R} \big(\mathcal{E}(B\bar{\tau}+\beta_1W_{1,\bar{\tau}})\big)^{-1}  =  \big(\mathcal{E}(B\bar{\tau}+\beta_1W_{1,\bar{\tau}})\big)^{-1} k^{k_m,c^*_{\nu}}_{\bar{\tau}}   =  k_m-\bar{c}\int_0^{\bar{\tau}}\frac{\de s}{\mathcal{E}(Bs+\beta_1W_{1,s})}.
		\]
		Consequently, for every $t>\bar{\tau}$, using \eqref{adm_control_cast} we have 
		\begin{align*}
			k^{k_m,c^*_{\nu}}_t &= \mathcal{E}(Bt+\beta_1 W_{1,t})\left(  k_m-\bar{c}\int_0^{\bar{\tau}}\frac{\de s}{\mathcal{E}(Bs+\beta_1W_{1,s})}\   -\int_{\bar{\tau}}^t\frac{c^\ast_{\nu, s}}{\mathcal{E}(Bs+\beta_1W_{1,s})}\de s\right)\\
			&=  \big(\mathcal{E}(B\bar{\tau}+\beta_1W_{1,\bar{\tau}})\big)^{-1}\mathcal{E}(Bt+\beta_1 W_{1,t})\left( \frac{\bar{c}}{R}-\frac{\nu}{B}(e^{-B\bar{\tau}}-e^{-Bt})\right),
		\end{align*}
		which shows that the constraints \eqref{constraints-kh} and \eqref{bdd_control} are satisfied for $\nu>0$ sufficiently small.\\
		
		We then fix $\nu^\ast>0$ such that $c^{\ast}\coloneqq c^{\ast}_{\nu^\ast}\in\mathcal{A}(k_m,h_m)$, and define the stopping times 
		$$
		\tau_\delta= \inf\left\{ t\ge 0 : |k^{k_m,{c}^{\ast}}_t-k_m|^2+|h^{h_m,{c}^{\ast}}_t-h_m|^2\ge \delta^2
		\right\}
		,
		\qquad \bar{\tau}_{\delta,\epsilon}=\tau_\delta\wedge \epsilon\wedge \bar{\tau}.$$ 
		Since $\bar{\tau}_{\delta,\epsilon}\le \epsilon$, $\mathbb{P}-$a.s., and $\big(k^{k_m,c^\ast}_t,h^{h_m,c^\ast}_t\big)\in \overline{B}_\delta(k_m,h_m)$ for every $t\le \bar{\tau}_{\delta,\epsilon}$, the functions $\psi,\nabla\psi,D^2\psi$ and the coefficients $a(\cdot,\cdot;\bar{c}),\Sigma$ are bounded on this compact set; we can therefore apply Dynkin's formula to the function $(t,k,h)\mapsto e^{-\theta t}\psi(k,h)$ to obtain, denoting by $h^*$ (resp., $k^*$) the process $h^{h_m,c^\ast}$(resp., $k^{k_m,c^\ast}$) to make the notation shorter, 
		\begin{align}\notag
			&\mathbb{E}\Big[e^{-\theta \bar{\tau}_{\delta,\epsilon}}\psi(k^*_{\bar{\tau}_{\delta,\epsilon}},h^*_{\bar{\tau}_{\delta,\epsilon}})\Big] - \psi(k_m,h_m)  =\mathbb{E}\bigg[\int_0^{\bar{\tau}_{\delta,\epsilon}}e^{-\theta t}\Big(-\theta \psi(k^*_t ,h^*_t ) + ((Bk^*_t -c^*_t )\partial_k\psi(k^*_t ,h^*_t )\\&\label{inter_super} \qquad + \rho(c^*_t-h^*_t) )\partial_h \psi(k^*_t ,h^*_t ) + \frac{1}{2}\beta_1^2(k^*_t) ^2\partial_{kk}\psi(k^*_t ,h^*_t ) +\frac{1}{2}\beta_2^2(h^*_t)^2 \partial_{hh}\psi(k^*_t ,h^*_t )\Big)\de t  \bigg]. 
		\end{align}
		We observe that, since  $\bar{\tau}_{\delta,\epsilon}\le \tau_\delta$,
		\[
		e^{-\theta \bar{\tau}_{\delta,\epsilon}} V(k^*_{\bar{\tau}_{\delta,\epsilon}} ,h^*_{\bar{\tau}_{\delta,\epsilon}})-V(k_m,h_m)\geq e^{-\theta \bar{\tau}_{\delta,\epsilon}} \psi(k^*_{\bar{\tau}_{\delta,\epsilon}} ,h^*_{\bar{\tau}_{\delta,\epsilon}} ) - \psi(k_m,h_m),\quad \mathbb{P}-\text{a.s.}
		\]
		Hence, by the dynamic programming principle in Proposition \ref{DPP}, we deduce from \eqref{inter_super} that  
		\begin{align*}
			0 \geq& \mathbb{E}\bigg[\int_0^{\bar{\tau}_{\delta,\epsilon}}e^{-\theta t}\bigg(-\theta \psi(k^*_t ,h^*_t ) + 
			\frac{1}{1-\sigma}\bigg(\frac{(h^*_t)^\gamma}{\bar{c}}\bigg)^{\sigma-1}
			+((Bk^*_t -\bar{c} )\partial_k\psi(k^*_t ,h^*_t )\\& \qquad + \rho(\bar{c}-h^*_t) )\partial_h \psi(k^*_t ,h^*_t ) + \frac{1}{2}\beta_1^2(k^*_t) ^2\partial_{kk}\psi(k^*_t ,h^*_t ) +\frac{1}{2}\beta_2^2(h^*_t)^2 \partial_{hh}\psi(k^*_t ,h^*_t )\bigg)\de t  \bigg],
		\end{align*}
		where we also use the fact that, by \eqref{adm_control_cast}, ${c}^\ast=\bar{c}$ on $[0,\bar{\tau}_{\delta,\epsilon}]$, as $\bar{\tau}_{\delta,\epsilon}\le \bar{\tau}.$
		Notice that $\tau_\delta>0$, $\mathbb{P}-$a.s., by the continuity of $h^*$ and $k^*$. Recalling that  the same holds true for $\bar{\tau}$, too, for $\mathbb{P}-$a.e. $\omega\in\Omega$, \[
		\text{$\bar{\tau}_{\delta,\epsilon}(\omega)=\epsilon$, $\quad$ when  $\epsilon$ is sufficiently small.}\]
		Therefore, dividing by $\epsilon$  the previous estimate and  taking the limit as $\epsilon\to0$, by the fundamental theorem of calculus and the dominated convergence theorem, 
		\begin{align*}
			0&\geq -\theta \psi(k_m ,h_m ) + \frac{1}{1-\sigma}\left(\frac{h_m^\gamma}{\bar{c}}\right)^{\sigma-1}  + ((Bk_m -\bar{c})\partial_k\psi(k_m ,h_m ) + \rho(\bar{c}-h_m) )\partial_h \psi(k_m ,h_m ) \\ & \qquad
			+ \frac{1}{2}\beta_1^2k_m ^2\partial_{kk}\psi(k_m ,h_m ) +\frac{1}{2}\beta_2^2h_m^2 \partial_{hh}\psi(k_m ,h_m )
			\\&=-\theta \psi(k_m ,h_m )+ 
			\mathcal{L}_{k_m,h_m}\psi
			+g(\bar{c};k_m,h_m,\nabla\psi(k_m ,h_m ))
			,
		\end{align*}
		where the last equality employs the notation in Section \ref{sec_HJB}, see after \eqref{def_Hmax} and \eqref{eq:HJB-1}.
		Since $\bar{c}$ is chosen arbitrarily in $(0,Rk_m)$ and $g(\cdot;k_m,h_m,\nabla\psi(k_m ,h_m ))$ is continuous at $Rk_m$, by taking the $\sup$ for $\bar{c}\in(0,Rk_m]$ we conclude that 
		\[
		\theta V(k_m,h_m)-H_{\text{max}}(k_m,h_m, \nabla\psi(k_m,h_m), D^2{\psi}(k_m,h_m))\ge0,
		\]
		hence $V$ is a viscosity supersolution of \eqref{eq:HJB-1}.
		
		\underline{\emph{Step \upperRomannumeral{2}}}.
		Suppose by contradiction that $V$ \emph{is not} a viscosity subsolution of \eqref{eq:HJB-1}. By Definition \ref{def_viscosity}, there exist  $\psi \in C^2(\mathbb{R}_{++}^2; \mathbb{R} )$, a local maximum point $(k_M,h_M) \in\mathbb{R}_{++}^2$ for $V-\psi$ and $\eta_1>0$ such that 
		\begin{equation*}
			\eta_1 < \theta V(k_M,h_M)- \mathcal{L}_{k_M,h_M}\psi
			-
			G(k_M,h_M, \nabla\psi(k_M,h_M)).
		\end{equation*}
		Without loss of generality, we assume $V(k_M,h_M)=\psi(k_M,h_M)$, hence there is a $\delta>0$ such that $V(k,h)\leq \psi(k,h)$, $(k,h)\in B_{\delta} (k_M,h_M)\subset \mathbb{R}_{++}^2$.
		By the continuity of $V$, see Theorem \ref{prop_cont}, since  $G$ in \eqref{max} is jointly continuous in its arguments and $\psi \in C^2$, we can find an $\epsilon\in(0,\delta)$ such that, for any $(k,h)\in B_{\epsilon}(k_M,h_M),$ recalling also the very definition of $G,$
		\begin{equation}\label{inter_subsolution}
			\frac{\eta_1}{2} < \theta V(k,h)- \mathcal{L}_{k,h}\psi +c (\partial_k\psi(k,h)-\rho\partial_h\psi(k,h))+\frac{1}{\sigma-1}\frac{h^{\gamma(\sigma-1)}}{c^{\sigma-1}},\quad  c\in(0,Rk].
		\end{equation}
		For any admissible control $\bar{c}\in\mathcal{A}(k_M,h_M)$, let
		$
		\tau_{\epsilon}=\inf\big\{t\geq 0: \big(k_t^{k_M,\bar{c}},h_t^{h_M,\bar{c}}\big)\notin B_{\epsilon}(k_M,h_M) \big\};
		$   
		notice that $\big(k^{k_M,\bar{c}},h^{h_M,\bar{c}}\big)\in B_\delta(k_M,h_M)$ on the random interval $[0,\tau_\epsilon]$ as $\epsilon<\delta$.
		Using $\tau_{\epsilon}$, from \eqref{inter_subsolution} we compute, denoting by $\bar{h}$ (resp., $\bar{k}$) the process $h^{h_M,\bar{c}}$ (resp., $k^{k_M,\bar{c}}$) to shorten the notation,
		\begin{align}\notag
			&\frac{\eta_1}{2}\mathbb{E}\bigg[ \int_{0}^{\tau_{\epsilon}\wedge 1} e^{-\theta t} \de t\bigg] - \mathbb{E}\bigg[ \int_{0}^{\tau_{\epsilon}\wedge 1} \frac{e^{-\theta t}}{\sigma-1}\left(\frac{\bar{h}^\gamma_t}{\bar{c}_t}\right)^{\sigma-1} \de t \bigg]  \\\label{eq:step2a}
			&\qquad <  \mathbb{E}\bigg[ \int_{0}^{\tau_{\epsilon}\wedge 1} e^{-\theta t}\Big( \theta \psi\big(\bar{k}_t,\bar{h}_t\big) - \mathcal{L}_{\bar{k}_t,\bar{h}_t}\psi  +\bar{c}_t (\partial_k\psi(\bar{k}_t,\bar{h}_t)-\rho\partial_h\psi(\bar{k}_t,\bar{h}_t))\Big) \de t \bigg].
		\end{align}
		Since $\mathbb{E}[\tau_\epsilon\wedge 1]\le 1 <\infty$, an application of Dynkin's formula as in \underline{\emph{Step \upperRomannumeral{1}}} yields 
		\begin{align}\notag
			& \mathbb{E}\Big[ e^{-\theta (\tau_{\epsilon}\wedge 1)}\psi\big(\bar{k}_{\tau_{\epsilon}\wedge 1},\bar{h}_{\tau_{\epsilon}\wedge 1}\big)  \Big] 
			- \psi(k_M,h_M) \\
			&\label{eq:step2D}\qquad
			= \mathbb{E} \bigg[ \int_0^{\tau_{\epsilon}\wedge 1}\!\!e^{-\theta t}\Big(\!-\theta \psi\big(\bar{k}_t,\bar{h}_t\big)   + \mathcal{L}_{\bar{k}_t,\bar{h}_t}\psi - \bar{c}_t(
			\partial_k\psi(\bar{k}_t,\bar{h}_t)-\rho\partial_h\psi(\bar{k}_t,\bar{h}_t))
			\Big) \de t \bigg].
		\end{align}
		Therefore, combining \eqref{eq:step2a} and \eqref{eq:step2D},
		\begin{align}
			V(k_M,h_M) >& \mathbb{E}\bigg[ \int_{0}^{\tau_{\epsilon}\wedge 1} \frac{e^{-\theta t}}{1-\sigma}\left(\frac{\bar{h}^\gamma_t}{\bar{c}_t}\right)^{\sigma-1}  \!\de t  + e^{-\theta (\tau_{ \epsilon}\wedge 1)} V\big(\bar{k}_{\tau_{\epsilon}\wedge 1},\bar{h}_{\tau_{ \epsilon}\wedge 1}\big)\bigg]\!  +  \frac{\eta_1}{2}\mathbb{E}\bigg[ \int_{0}^{\tau_{\epsilon}\wedge 1} \!\!e^{-\theta t} \de t\bigg]. \label{eq:step2T}
		\end{align}
		We now focus on the last term in \eqref{eq:step2T}. Inspired by \cite[Equation (2.49)]{choulli2003diffusion}, we define the $C^2-$function $w\colon \overline{B}_\epsilon(k_M,h_M)\to (-\infty, 0]$ by 
		\begin{equation*}
			w(k,h)=\frac{1}{C} \big((k-k_M)^2+(h-h_M)^2-\epsilon^2\big),\quad (k,h)\in  \overline{B}_\epsilon(k_M,h_M),
		\end{equation*}
		where $C>1$ is the  constant
		\begin{align*}
			C\coloneqq 1\vee \sup_{(k,h)\in \overline{B}_\epsilon(k_M,h_M)}\sup_{c\in (0,Rk]}
			\Big(&2|Bk-c||k-k_M| + 2\rho|c-h||h-h_M|
			+ \beta_1^2k^2 +\beta_2^2h^2 \\&
			+\theta(
			|k-k_M|^2+|h-h_M|^2+\epsilon^2)\Big)
			.
		\end{align*}
		Thanks to this choice of $C$, employing again Dynkin's formula with the function $(t,k,h)\mapsto e^{-\theta t}w(k,h)$, we infer that 
		\begin{align}
			\notag&\mathbb{E}\Big[  e^{-\theta (\tau_\epsilon\wedge 1)} w(\bar{k}_{\tau_\epsilon\wedge 1},\bar{h}_{\tau_\epsilon\wedge 1}) \Big] - 
			w(k_M,h_M)
			\\&\qquad =\frac1C
			\mathbb{E}\bigg[
			\int_0^{\tau_\epsilon\wedge 1}e^{-\theta t}\Big(-\theta w(\bar{k}_t,\bar{h}_t)+2 \Big[\begin{smallmatrix}
				B\bar{k}_t-\bar{c}_t \\ \rho(\bar{c}_t-\bar{h}_t)
			\end{smallmatrix}\Big]^{\top}
			\Big[\begin{smallmatrix}
				\bar{k}_t-k_M \\ \bar{h}_t-h_M
			\end{smallmatrix}\Big]
			+ \beta_1^2\bar{k}_t^2 +\beta_2^2\bar{h}_t^2 
			\Big)
			\de t
			\bigg]\notag
			\\&\label{inter_sub3}\qquad \le \mathbb{E}\bigg[\int_0^{\tau_\epsilon\wedge 1}e^{-\theta t}\de t  \bigg]
			.
		\end{align}
		Since $w\equiv 0$ on $\partial B_\epsilon(k_M,h_M)$ and $w\ge -\frac1C\epsilon^2$ in its domain, by the continuity of $\bar{k}$ and $\bar{h}$ we have 
		\begin{align*}
			\mathbb{E}\Big[  e^{-\theta (\tau_\epsilon\wedge 1)} w(\bar{k}_{\tau_\epsilon\wedge 1},\bar{h}_{\tau_\epsilon\wedge 1}) \Big] 
			&=
			\mathbb{E}\Big[ e^{-\theta{\tau_\epsilon}}w(\bar{k}_{\tau_\epsilon},\bar{h}_{\tau_\epsilon})1_{\{\tau_\epsilon\le1\}}\Big]
			+
			\mathbb{E}\Big[ e^{-\theta}w(\bar{k}_1,\bar{h}_1)1_{\{\tau_\epsilon>1\}}\Big]
			\\&
			\ge 0-\frac1C\epsilon^2e^{-\theta}\mathbb{P}(\tau_\epsilon>1)\ge -\frac1C\epsilon^2 e^{-\theta}
			.
		\end{align*} Consequently,
		noticing also that $w(k_M,h_M)=-\frac1C\epsilon^2$, \eqref{inter_sub3} gives 
		$
		\mathbb{E} \big[\int_{0}^{\tau_\epsilon\wedge 1} e^{-\theta t} \de t\big] \ge \frac1C\epsilon^2(1-e^{-\theta})\eqqcolon\eta_2>0.
		$
		Thus, \eqref{eq:step2T} implies that 
		\begin{equation*}
			V(k_M,h_M) > \mathbb{E}\bigg[ \int_{0}^{\tau_{\epsilon}\wedge 1} \frac{e^{-\theta t}}{1-\sigma}\left(\frac{\bar{h}^\gamma_t}{\bar{c}_t}\right)^{\sigma-1}  \de t  + e^{-\theta (\tau_{\epsilon}\wedge 1)} V\big(\bar{k}_{\tau_{ \epsilon}\wedge 1},\bar{h}_{\tau_{ \epsilon}\wedge 1}\big)\bigg]  +  \frac{\eta_1}{2}\eta_2.
		\end{equation*}
		We recall that $\bar{c}\in\mathcal{A}(k_M,h_M)$ is arbitrary and observe that the constant $\eta_2>0$ is independent of $\bar{c}$. Therefore, taking the supremum
		over $\mathcal{A}(k_M,h_M)$ in the previous estimate, we conclude that
		\[
		V(k_M,h_M) > \sup_{c\in \mathcal{A}(k_M,h_M)}\mathbb{E}\bigg[ \int_{0}^{\tau_{\epsilon}\wedge 1} \frac{e^{-\theta t}}{1-\sigma}\left(\frac{({h^{h_M,c}_t})^\gamma}{{c}_t}\right)^{\sigma-1}  \de t  + e^{-\theta (\tau_{\epsilon}\wedge 1)} V\big(k^{k_M,c}_{\tau_{ \epsilon}\wedge 1},h^{h_M,c}_{\tau_{ \epsilon}\wedge 1}\big)\bigg]  .
		\]
		This violates the dynamic programming principle in Proposition \ref{DPP}. It follows that $V$ is a viscosity subsolution of \eqref{eq:HJB-1}, and the proof is complete.
	\end{proof}

	\subsection{$V$ is a Classical Solution of the HJB Equation}\label{sec_classical}
	
	We recall the definition of classical solution of the HJB equation \eqref{eq:HJB-1}. 
	\begin{definition}\label{def:classical}
		Let $D$ be an open subset of $\mathbb{R}^{2}_{++}$. A function $v\in C^{2}(D;\mathbb{R})$ is called a \emph{classical solution} of \eqref{eq:HJB-1} if 
		\[
		\theta v(k,h) = H_{\text{max}}(k,h,\nabla v(k,h), D^2v(k,h)),\quad (k,h)\in D.
		\]
	\end{definition}
	In the next theorem, we combine the regularity properties of $H_\text{max}$ obtained in Section \ref{sec_HJB} with fundamental results in \cite{crandall1992user,safonov1989classical} to show that the value function $V$ is a  classical solution of \eqref{eq:HJB-1} in $\mathbb{R}^2_{++}$. This also proves that  $V$ is a regular function in $C^2(\mathbb{R}^2_{++}; \mathbb{R})$.
	\begin{theorem}\label{classical_sol}
		The value function $V$ satisfies the HJB equation \eqref{eq:HJB-1} in $\mathbb{R}_{++}^2$ in the classical sense.
	\end{theorem}
	
	\begin{proof}
		Fix  $(k_0,h_0)\in\mathbb{R}^2_{++}$ and consider $\delta_0>0$ such that $\overline{B}_{\delta_0}(k_0,h_0)\subset \mathbb{R}^2_{++}$.  Thanks to Lemma \ref{lemm:safonov} and Theorem \ref{prop_cont}, we can apply Theorem 1.1 in \cite{safonov1989classical} to state that, for some $\alpha\in(0,1),$ the HJB equation \eqref{eq:HJB-1} has a unique classical solution $v$ in $B_{\delta_0}(k_0,h_0)$ in the class 
		$$
		C^{2+\alpha}(B_{\delta_0}(k_0,h_0);\mathbb{R}) \cap C(\overline{B}_{\delta_0}(k_0,h_0);\mathbb{R})
		$$
		such that $v=V$ on $\partial B_{\delta_0}(k_0,h_0)$. 
		In particular, $v$ is also a (continuous) viscosity solution of \eqref{eq:HJB-1} in $B_{\delta_0}(k_0,h_0)$. Therefore, by Theorem 3.3 in \cite{crandall1992user}, which can be applied by Lemma \ref{lemm:userguide} and Theorem \ref{prop:VisVisc},  we deduce that 
		$$ 
		V=v\quad \text{ on }\quad \overline{B}_{\delta_0}(k_0,h_0). 
		$$  
		Since $(k_0,h_0)$ is arbitrary, the proof is complete. 
	\end{proof}
	\begin{remark}
		From the proof of Theorem \ref{classical_sol}, it follows that for every open bounded set $D$ such that $\overline{D}\subset \mathbb{R}^2_{++}$,  the value function  $V\in C^{2+\alpha}(D;\mathbb{R})$, for some $\alpha\in(0,1)$ possibly dependent on $D$.
	\end{remark}

	\section{Verification Theorem}\label{sec_verification}
	
	In this section, we  provide a sufficient condition for optimality of controls using the HJB equation. Moreover, we establish the existence of a unique optimal control, expressed in feedback form. The analysis will focus on controls with finite performance $J$, denoted by 
	\begin{equation}\label{def:Acirc}
		\mathcal{A}^{\circ}(k_0,h_0)\coloneqq\big\{c\in\mathcal{A}(k_0,h_0):\,J(k_0,h_0;c)>-\infty\big\},\quad (k_0,h_0)\in \mathbb{R}^2_{++}.
	\end{equation} 
	Evidently, $\mathcal{A}^{\circ}(k_0,h_0)\neq \emptyset$ by Proposition \ref{prop:lbound} and the value function $V$ in \eqref{def:valuefunction} coincides with 
	$    V(k_0,h_0)=\sup_{c\in\mathcal{A^{\circ}}(k_0,h_0)}  J(k_0,h_0;c)$.
	
	For every $p,q>0$, define the function $\Phi_{p,q}\colon\mathbb{R}_{++}^2\to\mathbb{R}_+$ by
	\begin{equation}\label{def:Phi1}
		\Phi_{p,q}(k,h)\coloneqq 1+k^p+h^p+k^{-q}+h^{-q}, \qquad (k,h)\in\mathbb{R}_{++}^2.
	\end{equation}
	We start with a preliminary lemma, whose proof is postponed to Appendix \ref{Appendix}. 
	
	\begin{lemma}\label{lem:moment_estimates_extended}
		Fix $(k_0,h_0)\in\mathbb{R}_{++}^2$, $t>0$ and $p,q>0$. Then, for every $c\in\mathcal{A}(k_0,h_0)$,   \begin{equation}\label{eq:Phi_moment_extended}
			\sup_{\tau\le t}  \mathbb{E}\left[\Phi_{p,q}(k^{k_0,c}_\tau,h^{h_0,c}_\tau)\right] <\infty, 
		\end{equation}
		where $\tau$ is a  stopping time  possibly depending on the control $c$.
	\end{lemma}
	
	We now consider the following growth assumption. 
	\begin{assumption}\label{assumptionVT}
		Given  a  function  $v: \mathbb{R}^2_{++}\to \mathbb{R}$, suppose that there exist $p,q>0$ such that 
		\begin{equation*}
			|v(k,h)|\leq C \Phi_{p,q}(k,h),\quad (k,h)\in\mathbb{R}^2_{++},
		\end{equation*}
		where $\Phi_{p,q}$ is defined in \eqref{def:Phi1} and $C>0$ is a positive constant.
	\end{assumption}

	\begin{theorem}[Smooth Verification, Sufficient Condition]\label{teo:sufficient} 
		Let $v:\mathbb{R}^2_{++}\to\mathbb{R}$ be a classical solution of the HJB equation \eqref{eq:HJB-1}  (see Definition \ref{def:classical}) which satisfies Assumption \ref{assumptionVT}. Suppose that, for every $(k_0,h_0)\in\mathbb{R}^2_{++}$ and every $c\in\mathcal{A}^{\circ}(k_0,h_0)$,  the following transversality condition holds: 
		\begin{equation}\label{trans_v} 
			\limsup_{T\to\infty}\mathbb{E}\Big[e^{-\theta T}v(k^{k_0,c}_T,h^{h_0,c}_T)\Big]=0.
		\end{equation}
		Then	
		\begin{equation}\label{eq_classicalvalue}
			v(k,h)\geq V(k,h),\quad (k,h)\in \mathbb{R}^2_{++},
		\end{equation}
		where $V$ is the value function in \eqref{def:valuefunction}. \vspace{1mm}\\
		In addition, given $(k_0,h_0)\in\mathbb{R}^2_{++}$, suppose that  $c^*\in \mathcal{A}^{{\circ}}(k_0,h_0)$ satisfies 
		\begin{equation}\label{eq:condition}
			c^*_t= \Gamma\big(k_t^{k_0,c^\ast},h_t^{h_0,c^\ast}, \nabla v (k_t^{k_0,c^\ast},h_t^{h_0,c^\ast})\big),\quad \text{$\mathbb{P}\otimes \de t-$a.e.,
			}
		\end{equation}
		where $\Gamma$ is defined in Lemma \ref{gen_le_max_argmax}. Then  $c^*$  is optimal at $(k_0,h_0)$, and $v(k_0,h_0)=V(k_0,h_0).$
	\end{theorem}
	
	\begin{proof}
		Fix $(k_0,h_0)\in \mathbb{R}^2_{++}$ and consider $c\in\mathcal{A}^{{\circ}}(k_0,h_0)$. Denote by  $k=k^{k_0,c}$ and $h=h^{h_0,c}$  the solutions of  \eqref{eq:k} and \eqref{eq:h}, respectively. For every $n\in\mathbb{N}$, define the stopping times
		\begin{equation}\label{def:taun_new}
			\tau_n\coloneqq\inf\big\{t\ge0:\,(k_t,h_t)\notin K_n\big\},\qquad K_n\coloneqq[1/n,n]^2;
		\end{equation} 
		notice that $\tau_n\to \infty$, $\mathbb{P}-$a.s. By It\^o's formula, for every $T>0$, letting $S_n=\tau_n\wedge T$, we have 
		\begin{align*}
			&e^{-\theta S_n}v(k_{S_n},h_{S_n})  = v(k_0,h_0) +\int_0^{S_n} e^{-\theta t}\Big( -\theta v(k_t,h_t) + \partial_kv(k_t,h_t)\big(Bk_t-c_t\big) \\
			& \qquad + \partial_h v(k_t,h_t)\rho(c_t-h_t)+\frac{1}{2}\mathrm{tr}\big(\Sigma (k_t,h_t)^\top D^2v(k_t,h_t)\Sigma(k_t,h_t)\big)\Big ) \de t \\
			& \qquad  +\beta_1 \int_0^{T}1_{\{t\le {S_n}\}} e^{-\theta t} k_t\partial_kv(k_t,h_t)\de W_{1,t} +\beta_2\int_0^T 1_{\{t\le {S_n}\}}e^{-\theta t} h_t\partial_hv(k_t,h_t)\de W_{2,t},\quad \mathbb{P}-\text{a.s.} 
		\end{align*}
		Since $v\in C^2(\mathbb{R}^2_{++};\mathbb{R})$, by  \eqref{def:taun_new}, taking  the expected value in the previous equation gives, by the martingale property of the stochastic integrals, 
		\begin{align*}
			&\mathbb{E}\Big[e^{-\theta {S_n}}v(k_{S_n},h_{S_n})\Big] + \mathbb{E}\bigg[\int_0^{S_n} \frac{e^{-\theta t}}{1-\sigma}\left(\frac{c_t}{h_t^\gamma}\right)^{1-\sigma} \de t \bigg]=
			v(k_0,h_0) + \mathbb{E}\bigg[\int_0^{S_n} e^{-\theta t}\Big( -\theta v(k_t,h_t) \\
			&\qquad + \partial_kv(k_t,h_t)\big(Bk_t-c_t\big) + \partial_h v(k_t,h_t)\rho(c_t-h_t) +\frac{1}{2}\mathrm{tr}\big(\Sigma (k_t,h_t)^\top D^2v(k_t,h_t)\Sigma(k_t,h_t)\big) \\
			& \qquad + \frac{1}{1-\sigma}\left(\frac{c_t}{h_t^\gamma}\right)^{1-\sigma} \bigg )\de t \bigg],\quad n\in \mathbb{N}.
		\end{align*}
		Since $v$ is a classical solution of \eqref{eq:HJB-1},  we write, for every $n\in\mathbb{N}$,
		\begin{align}
			&\notag\mathbb{E}\Big[e^{-\theta {S_n}}v(k_{S_n},h_{S_n})\Big] +\mathbb{E}\bigg[\int_0^{S_n} \frac{e^{-\theta t}}{1-\sigma}\left(\frac{c_t}{h_t^\gamma}\right)^{1-\sigma} \de t \bigg]= v(k_0,h_0) \\
			&\notag\qquad+ \mathbb{E}\bigg[\int_0^{S_n} e^{-\theta t}\Big( g\big(c_t;k_t, h_t,\nabla v(k_t,h_t)\big) -G\big(k_t,h_t,\nabla v(k_t,h_t) \big) \Big )\de t \bigg].
		\end{align}
		By Lemma \ref{lem:moment_estimates_extended}, Assumption \ref{assumptionVT} and Vitali's convergence theorem, we have $\mathbb{E}[e^{-\theta {S_n}}v(k_{S_n},h_{S_n})] \to \mathbb{E}[e^{-\theta {T}}v(k_{T},h_{T})] $  as $n\to\infty$.
		We also observe that, by the definition of  $G$ (see Lemma \ref{gen_le_max_argmax}), the last term on the right-hand side of the previous equation is nonpositive. Hence, by the monotone convergence theorem
		\begin{align}\label{comparison_dopo}
			&\notag\mathbb{E}\Big[e^{-\theta {T}}v(k_{T},h_{T})\Big] +\mathbb{E}\bigg[\int_0^{T} \frac{e^{-\theta t}}{1-\sigma}\left(\frac{c_t}{h_t^\gamma}\right)^{1-\sigma} \de t \bigg]= v(k_0,h_0) \\
			&\qquad+ \mathbb{E}\bigg[\int_0^{T} e^{-\theta t}\Big( g\big(c_t;k_t, h_t,\nabla v(k_t,h_t)\big) -G\big(k_t,h_t,\nabla v(k_t,h_t) \big) \Big )\de t \bigg]. 
		\end{align}
		By the hypothesis in \eqref{trans_v}, again by the monotone convergence theorem as $T\to\infty $, 
		\begin{equation}\label{eq:fidentity}
			J(k_0,h_0;c)=v(k_0,h_0)+ \mathbb{E}\bigg[\int_0^{\infty} e^{-\theta t}\Big( g\big(c_t;k_t, h_t,\nabla v(k_t,h_t)\big) -G\big(k_t,h_t,\nabla v(k_t,h_t) \big) \Big )\de t \bigg].  
		\end{equation}
		In particular,
		$ 
		v(k_0,h_0)\ge J(k_0,h_0;c).
		$
		Since $c$ is arbitrary,  we can take the {supremum} over $\mathcal{A}^{\circ}(k_0,h_0)$ in this estimate to establish \eqref{eq_classicalvalue}, as $(k_0,h_0)$ is a generic point in $\mathbb{R}^2_{++}$. \vspace{1mm}

		If we have an admissible control $c^\ast\in \mathcal{A}^{{\circ}}(k_0,h_0)$ satisfying  \eqref{eq:condition}, then, letting $k^\ast=k^{k_0,c^\ast}$ and $h^\ast=h^{h_0,c^\ast}$,
		\begin{align*}
			\mathbb{E}\bigg[\int_0^{\infty} e^{-\theta t}\Big( g\big(c^\ast_t;k^\ast_t, h^\ast_t,\nabla v(k^\ast_t,h^\ast_t)\big) -G\big(k^\ast_t,h^\ast_t,\nabla v(k^\ast_t,h^\ast_t) \big) \Big )\de t \bigg]=0.
		\end{align*}
		Consequently, by \eqref{eq_classicalvalue} (which we have just proved) and \eqref{eq:fidentity}, $$V(k_0,h_0)\ge J(k_0,h_0;c^\ast)=v(k_0,h_0) \ge V(k_0,h_0),$$ which completes the proof.
	\end{proof}
	\begin{remark}
		Equation \eqref{eq:fidentity} in the proof of Theorem \ref{teo:sufficient} is the so-called \emph{fundamental identity} for the optimal control problem. Note also that the conclusion \eqref{eq_classicalvalue} of Theorem \ref{teo:sufficient} holds when $v$ is  only a classical \emph{supersolution} of the HJB equation \eqref{eq:HJB-1}.
	\end{remark}

	
	\begin{proposition}\label{prop:VsatisfiesA}
		Let Assumption \ref{assumption1} hold. Then the value function $V$ defined in \eqref{def:valuefunction} satisfies Assumption \ref{assumptionVT} and the transversality condition \eqref{trans_v}.
	\end{proposition}
	\begin{proof}
		Since $V\le0$ on $\mathbb{R}^2_{++}$, by Proposition \ref{prop:lbound} there is a constant $C>0$ such that 
		\begin{equation}\label{eq:absV_bound}
			|V(k,h)|\le C\bigg(\frac{h^{\gamma(\sigma-1)}}{k^{\sigma-1}} +\frac{1}{k^{(1-\gamma)(\sigma-1)}}\bigg),\quad (k,h)\in\mathbb{R}^2_{++}.
		\end{equation}
		Let $p=2\gamma(\sigma-1)$ and $ q= 2(\sigma-1);$ notice that $p,q>0$  as  $\sigma>1$. Young's inequality applied with $h^{\gamma(\sigma-1)}$ and $k^{-(\sigma-1)}$ gives
		\[
		\frac{h^{\gamma(\sigma-1)}}{k^{\sigma-1}} \le\frac12\Big(h^{2\gamma(\sigma-1)}+k^{-2(\sigma-1)}\Big)\le\frac12\Big(1+h^p+k^{-q}\Big).
		\]
		Moreover, since $(1-\gamma)(\sigma-1)<\sigma-1<q$ because $\gamma\in(0,1]$,
		$
		{k^{-(1-\gamma)(\sigma-1)}}\le 1+k^{-q}.$
		Plugging the last two estimates into \eqref{eq:absV_bound} and recalling \eqref{def:Phi1}, we obtain  Assumption \ref{assumptionVT} for $V$.
		
		As for \eqref{trans_v}, fix $(k_0,h_0)\in\mathbb{R}^2_{++}$ and $c\in\mathcal{A}^{\circ}(k_0,h_0)$. Since the controlled habit dynamics \eqref{eq:h} are continuous with regular drift and diffusion coefficients, by the Kolmogorov--Chentsov criterion they generate a continuous stochastic flow with respect to the initial time and space and the evolution time. Recalling that  $\mathbb{F}$ is the minimal augmented Brownian filtration generated by $W_1$ and $W_2$, we can use the flow property of $h^{h_0,c}$ and argue as in  \cite[Subsections 6.1 and 6.2]{bondi2023regular}, which treats more general dynamics including jumps, to infer that 
		\begin{equation*}
			\frac{1}{1-\sigma}\mathbb{E}\bigg[\int_{T}^{\infty}e^{-\theta(s-T)}\bigg(\frac{(h^{h_0,c}_s)^{\gamma}}{c_s}\bigg)^{\sigma-1}\,\de s\,\bigg|\,\mathcal{F}_T\bigg] \le V(k^{k_0,c}_T,h^{h_0,c}_T),\quad \mathbb{P}-\text{a.s., } T\ge0 .
		\end{equation*}
		The previous inequality follows also from the pseudo-Markov property in \cite{claisse2016pseudo}. 
		Taking expectations, by the monotone convergence theorem \begin{equation*}
			0\le\mathbb{E}\big[e^{-\theta T}\big|V(k^{k_0,c}_T,h^{h_0,c}_T)\big|\big] \le \big|J(k_0,h_0;c)\big|-\frac{1}{\sigma-1}\mathbb{E}\bigg[\int_0^{T}e^{-\theta s}\bigg(\frac{(h^{h_0,c}_s)^{\gamma}}{c_s}\bigg)^{\sigma-1}\,\de s\bigg] \underset{T\to\infty}{\longrightarrow}0,
		\end{equation*}
		which shows \eqref{trans_v} and completes the proof.
	\end{proof}
	We conclude this paper by establishing the existence of a unique optimal control for our problem. This control will be given in feedback form.
	
	\begin{theorem}\label{ex_uni_CLE}
		Define the feedback map $\Gamma_V\colon \mathbb{R}_{++}^{2}\to \mathbb{R}_{+}$ by $\Gamma_V(k,h)=\Gamma(k,h,\nabla V(k,h)),$ where $\Gamma$ is given in Lemma \ref{gen_le_max_argmax} and $V$ is the value function. Then for every $(k_0,h_0)\in\mathbb{R}^2_{++}$ the closed loop equations (CLEs)
		\begin{align}\label{eq:kCL}
			&\de k^{k_0,\Gamma_V}_t
			=
			\big(Bk^{k_0,\Gamma_V}_t
			-\Gamma_V(k^{k_0,\Gamma_V}_t,h^{h_0,\Gamma_V}_t)\big)\de t
			+\beta_1k^{k_0,\Gamma_V}_t\de W_{1,t},
			\quad k_0>0,\\
			\label{eq:hCL}
			&\de h^{h_0,\Gamma_V}_t
			=
			\rho\big(\Gamma_V(k^{k_0,\Gamma_V}_t,h^{h_0,\Gamma_V}_t)
			-h^{h_0,\Gamma_V}_t\big)\de t
			+\beta_2h^{h_0,\Gamma_V}_t\de W_{2,t},
			\quad h_0>0,
		\end{align}
		admit a pathwise unique solution  with strictly positive components.
		
		Furthermore, the process $c^\ast$ defined by $c^\ast_t = \Gamma_V(k^{k_0,\Gamma_V}_t,h^{h_0,\Gamma_V}_t)$ belongs to $ \mathcal{A}^{\circ}(k_0,h_0)$ and is the unique optimal control at $(k_0,h_0)\in \mathbb{R}^2_{++}.$
	\end{theorem}

	\begin{proof}
		From the linear growth of $\Gamma$ established in Lemma \ref{gen_le_max_argmax}, we infer that the drifts of the CLEs have linear growth on $\mathbb{R}^2_{++}$. Moreover, by the regularity of $V$ established in Theorem \ref{classical_sol} and the explicit expression of $\Gamma$ in \eqref{arg_max}, it follows immediately that the drifts are also locally Lipschitz continuous on $\mathbb{R}^2_{++}$. Since the diffusion coefficients are linear, by, e.g., \cite[Theorems 3.1 and 3.2, Chapter 5]{durrett2018stochastic} and a standard localization argument, we conclude that the CLEs admit a pathwise unique solution $(k^\ast,h^\ast)$ up to $\tau^*=\lim_{n\to\infty}\inf\{t\ge0, k^\ast_t\wedge h^\ast_t =1/n\}$. 
		We now focus on the positivity of $(k^\ast,h^\ast)$. 
		Let  $c^\ast_t=1_{\{t< \tau^*\}}\Gamma_V(k^\ast_t,h^\ast_t)$; since   $0<c^\ast_t\le Rk^\ast_t$ on $[0,\tau^\ast)$, defining by 
		$
		N_t=\mathcal{E}\big((B-R)t+\beta_1W_{1,t}\big)^{-1}$  for $t\ge 0$, 
		by the stochastic product rule $$\de\big(k^\ast_tN_t\big) =\Big(N_t\big(Bk^\ast_t-c^\ast_t\big)+k^\ast_tN_t\big((R-B)+\beta_1^2\big)  -\beta_1^2k^\ast_tN_t\Big)\de t =N_t\big(Rk^\ast_t-c^\ast_t\big)\de t\ge0,\quad t\in [0,\tau^\ast),$$
		hence $k_t^\ast\ge k_0N_t^{-1}$ for $t\in[0,\tau^\ast)$. Analogously, letting $	H_t=\mathcal{E}\big(-\rho t+ \beta_2W_{2,t}\big)^{-1},
		$  
		\[
		\de(h^\ast_tH_t)=\Big(\rho H_t(c^\ast_t-h^\ast_t)+(\rho+\beta_2^2) h^\ast_t H_t-\beta_2^2h^\ast_t H_t\Big)\de t = \rho H_tc^\ast_t\,\de t\ge0,\quad t\in[0,\tau^\ast),
		\]
		hence 
		$h_t^\ast\ge h_0H_t^{-1}$ for $t\in[0,\tau^\ast)$. Therefore we infer that $\tau^\ast=\infty,\,\mathbb{P}-$a.s. This proves that $(k^\ast, h^\ast)$ is the pathwise unique solution of \eqref{eq:kCL}-\eqref{eq:hCL} on $[0,\infty).$
		
		The previous arguments show that $c^\ast=\Gamma_V(k^\ast,h^\ast)\in \mathcal{A}(k_0,h_0).$ In fact, $J(k_0,h_0;c^\ast)>-\infty$, so $c^\ast \in\mathcal{A}^\circ(k_0,h_0)$. To see this, one can proceed as in the proof of Theorem \ref{teo:sufficient} (with $V$ instead of $v$) to infer that,  recalling that $V$ is nonpositive and finite by Proposition \ref{prop:lbound}, (cf. \eqref{comparison_dopo})
		\begin{align*}
			& J(k_0,h_0;c^\ast)=\lim_{T\to\infty}\mathbb{E}\bigg[\int_0^{T} \frac{e^{-\theta t}}{1-\sigma}\left(\frac{c^\ast_t}{(h^\ast_t)^\gamma}\right)^{1-\sigma} \de t \bigg]\\&\qquad \ge
			\limsup_{T\to \infty} \bigg(\mathbb{E}\Big[e^{-\theta {T}}V(k^\ast_{T},h^\ast_{T})\Big] +\mathbb{E}\bigg[\int_0^{T} \frac{e^{-\theta t}}{1-\sigma}\left(\frac{c^\ast_t}{(h^\ast_t)^\gamma}\right)^{1-\sigma} \de t \bigg]\bigg)= V(k_0,h_0)>-\infty,
		\end{align*}
		where for the last equality we use the definition of $c^\ast$. This estimate shows that $c^\ast$ is optimal at $(k_0,h_0). $ The same conclusion follows from the general sufficient condition in  Theorem \ref{teo:sufficient}, that can be applied by Theorem \ref{classical_sol}, Proposition \ref{prop:VsatisfiesA} and $c^\ast \in \mathcal{A}^\circ(k_0,h_0)$.
		
		As for the uniqueness of $c^\ast$, we consider a general optimal control $c_1^\ast\in \mathcal{A}^\circ(k_0,h_0)$, that is, 
		$J(k_0,h_0;c_1^\ast)=V(k_0,h_0)>-\infty$. The fundamental identity in \eqref{eq:fidentity} for the value function $V$ then yields
		\begin{align*}
			\mathbb{E}\bigg[\int_0^{\infty} \!\!e^{-\theta t}\Big(& g\big(c^\ast_{1,t};k^{k_0,c_1^\ast}_t, h^{h_0,c_1^\ast}_t,\nabla V(k^{k_0,c_1^\ast}_t,h^{h_0,c_1^\ast}_t)\big) -G\big(k^{k_0,c_1^\ast}_t,h^{h_0,c_1^\ast}_t,\!\nabla V(k^{k_0,c_1^\ast}_t,h^{h_0,c_1^\ast}_t) \big) \Big )\de t \bigg]=0,
		\end{align*}
		whence $c_{1,t}^\ast= \Gamma_V(k^{k_0,c_1^\ast}_t,h^{h_0,c_1^\ast}_t), \,\mathbb{P}\otimes \de t-$a.e. Consequently, the couple $(k^{k_0,c_1^\ast},h^{h_0,c_1^\ast})$ satisfies the CLEs \eqref{eq:kCL}-\eqref{eq:hCL}. Since they admit a unique solution (previously denoted by $(k^\ast,h^\ast)$), we conclude that $c_1^\ast=c^\ast,\,\mathbb{P}\otimes \de t-$a.e., and the proof is complete.
	\end{proof}


	\appendix

	\section{Proof of Lemmas \ref{lemm:safonov}, \ref{lemm:userguide} and \ref{lem:moment_estimates_extended} }\label{Appendix}


	\begin{myproof}{Lemma \ref{lemm:safonov}}
		We denote by $C>0$ a positive constant possibly depending on $\overline{D}\subset \mathbb{R}^2_{++}$ and on the parameters of the model, allowed to change from line to line.  
		Fix $Q=(Q_{ij})_{i,j}\in\mathcal{S}^2, \,(k,h)\in \overline{D},\, \pi=\left[\begin{smallmatrix}
			\pi_k\\\pi_h
		\end{smallmatrix}\right],
		\bar{\pi}=\left[\begin{smallmatrix}
			\bar{\pi}_k\\\bar{\pi}_h
		\end{smallmatrix}\right]\in\mathbb{R}^2$ and $v,\bar{v}\in\mathbb{R}.$ 
		To prove \eqref{ineq:Saf2} we distinguish two cases: (i)  $\pi$  and $\overline{p}$ satisfy the same inequality in \eqref{max}; (ii) $\pi$ satisfies the first inequality in  \eqref{max} and $\bar{\pi}$ the second.\vspace{1mm}\\
		\underline{\emph{Case (i)}}. We consider: (I)
		$ \max\{\pi_k-\rho \pi_h, \bar{\pi}_k-\rho \bar{\pi}_h\}
		\le 
		h^{\gamma(\sigma-1)}(R k)^{-\sigma} $; (II)
		$ \min\{\pi_k-\rho \pi_h, \bar{\pi}_k-\rho \bar{\pi}_h
		\}> 
		h^{\gamma(\sigma-1)}(R k)^{-\sigma}$.
		Notice that, by the continuity of $G$ in \eqref{max}, in Case (II) we can also consider $\ge$ instead of $>$. Since $D$ is bounded, straightforward computations based on \eqref{def_Hmax} and \eqref{def_Ftilde} yield
		\begin{align*}
			&  |\tilde{F}(Q, \pi, v, (k, h)) - \tilde{F}(Q, \bar{\pi}, \bar{v}, (k, h))| \leq  \theta |v - \bar{v}| + C (|\pi_k - \bar{\pi}_k| + |\pi_h - \bar{\pi}_h|)\\
			&\qquad + \begin{cases}  Rk (|\pi_k - \bar{\pi}_k|+\rho|\pi_h-\bar{\pi}_h| ),& \mbox{if (I) holds}, \\
				(1+\frac{1}{\sigma-1})h^{\gamma(1-\frac{1}{\sigma})} \big|(\pi_k - \rho \pi_h)^{1-\frac{1}{\sigma}} - (\bar{\pi}_k - \rho \bar{\pi}_h)^{1-\frac{1}{\sigma}}\big|,& \mbox{if (II) holds}. 
			\end{cases}
		\end{align*}
		Since $x^{1-\frac{1}{\sigma}}$ is Lipschitz-continuous on $(\epsilon,\infty),$ for any $\epsilon>0$, and  $\overline{D}\subset \mathbb{R}^2_{++}$,  \eqref{ineq:Saf2} follows.\\
		\underline{\emph{Case (ii).}} Suppose that  $\pi_k - \rho \pi_h \leq h^{\gamma(\sigma - 1)}(Rk)^{-\sigma} $ and $\bar{\pi}_k - \rho \bar{\pi}_h > h^{\gamma(\sigma - 1)}(Rk)^{-\sigma} $. 
		Since 
		\begin{align*}
			0&\le  \max\{h^{\gamma(\sigma - 1)}(Rk)^{-\sigma}-(\pi_k-\rho \pi_h),(\bar{\pi}_k-\rho \bar{\pi}_h) -h^{\gamma(\sigma - 1)}(Rk)^{-\sigma}
			\}\\&
			\le (\bar{\pi}_k-\rho \bar{\pi}_h)-(\pi_k-\rho \pi_h)
			\le C(|\pi_k-\bar{\pi}_k|+|\pi_h-\bar{\pi}_h|),
		\end{align*}
		\eqref{ineq:Saf2} follows from \eqref{max} by the triangle inequality similarly to \underline{\emph{Case (i)}}.\vspace{2mm}
		
		As regards \eqref{eq_saf}, by the explicit expressions in \eqref{def_Hmax}, \eqref{max} and \eqref{def_Ftilde}, since  $\overline{D}\subset \mathbb{R}^2_{++}$, similarly to \underline{\emph{Case (ii)}} above we see that 
		$C_1=C\big(1+\sum_{i,j=1}^2 |Q_{ij}| + |\pi_k| + |\pi_h| + |v|\big)$
		is a Lipschitz-continuity constant for $\tilde{F}(Q,\pi,v,\cdot)$  on ${D}$. Since $D$ is bounded, for any $\alpha\in(0,1)$, $\tilde{F}(Q,\pi,v,\cdot)$ is then $\alpha-$H\"older continuous on $D$ with a constant given by a multiple of $C_1$. Therefore \eqref{eq_saf} follows and  the proof is complete.  
	\end{myproof}
	\vspace{2mm}
	
	\begin{myproof}{Lemma \ref{lemm:userguide}} As in  the proof of Lemma \ref{lemm:safonov}, we denote by $C>0$ a positive constant possibly depending on $D\subset \mathbb{R}^2_{++}$ and on the parameters of the model, allowed to change from line to line. 
		Fix $\alpha>0$, $(k,h),\,(\bar{k},\bar{h})\in D,\,v\in\mathbb{R}$ and $Q,\bar{Q}\in\mathcal{S}^2$ such that \eqref{ineq:QIQ} is fulfilled. Based on \eqref{max}, we distinguish the two cases (A) and (B).
		\begin{itemize}
			\item[(A)] One of the following holds: 
			\begin{itemize}
				\item[($A_1$)]$\alpha({k}-\bar{k}-\rho({h}-\bar{h}))\leq \bar{h}^{\gamma(\sigma-1)}(R\bar{k})^{-\sigma}\quad \!$ and $\quad \!\alpha({k}-\bar{k}-\rho({h}-\bar{h}))\leq h^{\gamma(\sigma-1)}(Rk)^{-\sigma}$;
				\item[($A_2$)]$\alpha({k}-\bar{k}-\rho({h}-\bar{h}))> \bar{h}^{\gamma(\sigma-1)}(R\bar{k})^{-\sigma}\quad \!$ and $\quad \!\alpha({k}-\bar{k}-\rho({h}-\bar{h}))> h^{\gamma(\sigma-1)}(Rk)^{-\sigma}.$
			\end{itemize}
			\item[(B)] One of the following holds:  
			\begin{itemize}
				\item[($B_1$)] $\alpha({k}-\bar{k}-\rho({h}-\bar{h}))> \bar{h}^{\gamma(\sigma-1)}(R\bar{k})^{-\sigma}\quad \!$ and $\quad \!\alpha({k}-\bar{k}-\rho({h}-\bar{h}))\leq h^{\gamma(\sigma-1)}(Rk)^{-\sigma}$;
				\item[($B_2$)] $\alpha({k}-\bar{k}-\rho({h}-\bar{h}))\leq \bar{h}^{\gamma(\sigma-1)}(R\bar{k})^{-\sigma}\quad \!$ and $\quad \!\alpha({k}-\bar{k}-\rho({h}-\bar{h}))> h^{\gamma(\sigma-1)}(Rk)^{-\sigma}$.
			\end{itemize}
		\end{itemize}
		In Case (A), writing $\Sigma$ (resp., $\bar{\Sigma}$) for the diagonal matrix $\Sigma(k,h)$ (resp., $\Sigma(\bar{k},\bar{h})$) to simplify the notation, we compute
		\begin{align*}
			&\tilde{F}(Q, \alpha(k-\bar{k}, h-\bar{h}), v,(k, h)) 
			-
			\tilde{F}(\bar{Q}, \alpha(k-\bar{k}, h-\bar{h}),v, (\bar{k}, \bar{h}))\\ 
			&\qquad = \alpha B(k-\bar{k})^2 -\rho\alpha(h-\bar{h})^2 +\frac{1}{2}\mathrm{tr}\big( {\Sigma}^\top{Q}{\Sigma}-\bar{\Sigma}^{\top}\bar{Q}\bar{\Sigma}\big)\\&\qquad\quad  +\begin{cases}
				-\alpha R (k-\bar{k})^2+\alpha R\rho(k-\bar{k})(h-\bar{h})+ \frac{1}{\sigma-1}\big((\frac{\bar{h}^{\gamma}}{R\bar{k}})^{\sigma-1}-(\frac{h^{\gamma}}{R k})^{\sigma-1}\big), & \mbox{ if } (A_1), \\
				\big(1+\frac{1}{\sigma-1}\big)\big( \alpha(k-\bar{k})-\rho\alpha(h-\bar{h}) \big)^{1-\frac{1}{\sigma}}\big(\bar{h}^{\gamma(1-\frac{1}{\sigma})}-h^{\gamma(1-\frac{1}{\sigma})}\big), & \mbox{ if } (A_2).
			\end{cases}
		\end{align*}
		The boundedness of $D$, \eqref{def_Hmax} and the fact that  $\overline{D}\subset \mathbb{R}^2_{++}$, together with the trivial inequality $2ab\le a^2+b^2,\,(a,b)\in\mathbb{R}^2,$  yield (cf. \underline{\emph{Case (i)}}\,(\upperRomannumeral{1}) in the proof of Lemma \ref{lemm:safonov})
		\begin{align}
			\notag& \tilde{F}(Q, \alpha(k-\bar{k}, h-\bar{h}), v,(k, h)) 
			-
			\tilde{F}(\bar{Q}, \alpha(k-\bar{k}, h-\bar{h}),v, (\bar{k}, \bar{h}))
			\\ 
			&\qquad \leq C\big(\alpha |(k-\bar{k},h-\bar{h})|^2 + |(k-\bar{k},h-\bar{h})|\big) +\frac{1}{2}\mathrm{tr}\big({Q}{\Sigma}^2-\bar{Q}\bar{\Sigma}^2\big).\label{casse(A)}
		\end{align}
		We now estimate the trace in \eqref{casse(A)}. Multiplying \eqref{ineq:QIQ} on the right by the positive semidefinite matrix 
		$$\begin{bmatrix}
			\Sigma^2 && \,\bar{\Sigma}\Sigma \\
			\bar{\Sigma} \Sigma&&\,\bar{\Sigma}^2
		\end{bmatrix}\in \mathcal{S}^4$$
		and then taking the trace preserve the inequality in \eqref{ineq:QIQ}. Then 
		\begin{equation*}
			\mathrm{tr}\big( Q\Sigma^2
			-\bar{Q}\bar{\Sigma}^2\big) 
			\leq 3\alpha \,
			\mathrm{tr}\big( (\Sigma -\bar{\Sigma})^2\big) \leq C\alpha|(k-\bar{k},h-\bar{h})|^2.
		\end{equation*}
		Plugging this bound into \eqref{casse(A)} gives \eqref{est_user}, completing the proof of Case (A).
		\vspace{2mm}
		\\
		By the continuity of $G$ in \eqref{max}, in Case $(A_2)$ we can also consider the inequalities with  $\ge$ instead of $>$. Thus, Case (B) can be readily deduced by the triangle inequality similarly to Case (A), see also \underline{\emph{Case (ii)}} and the argument for \eqref{eq_saf} in the proof of Lemma \ref{lemm:safonov}. The proof is now complete.
	\end{myproof}\vspace{2mm}
	
	\begin{myproof}{Lemma \ref{lem:moment_estimates_extended}}
		Since $x^p\le 1+x$ for $x\ge 0$ and $p\in (0,1)$, it is enough to consider $p\ge1$ and $q>0$.  We denote by $C>0$ a constant depending only on $p,q$  and the parameters of the model; $C$ is  allowed to change from line to line.
		Fix $c\in\mathcal{A}(k_0,h_0)$. We first establish \eqref{eq:Phi_moment_extended} for a deterministic time $t>0$, i.e., $\mathbb{E}[\Phi_{p,q}(k^{k_0,c}_t,h^{h_0,c}_t)]<\infty$. We start with estimating
		the positive $p-$moments. Since $c>0$, \eqref{expl_eq_k} gives
		$
		k_t
		\leq
		k_0\mathcal{E}(Bt+\beta_1W_{1,t}),
		\, t\geq0,
		\,\mathbb{P}-\text{a.s.}
		$
		Therefore
		\begin{equation}
			\mathbb{E}[k_t^p]
			\leq
			k_0^pe^{pBt} e^{\frac{1}{2}p(p-1)\beta_1^2t}
			\mathbb{E}\left[\mathcal{E}(p\beta_1W_{1,t})\right]
			\le 
			k_0^p
			\exp\left\{
			p\left(B+\frac{(p-1)^+}{2}\beta_1^2\right)t
			\right\}.
			\label{eq:proof_positive_k}
		\end{equation}
		As for $h$, from the preceding
		pathwise bound for $k$ we obtain, using $c_t\leq Rk_t$,   $(x+y)^p\leq 2^{p-1}(x^p+y^p)$, \eqref{expl_eq_h}  and Jensen's inequality, and recalling that $\rho>0$,
		\begin{align}
			h_t^p
			&\leq 2^{p-1}e^{-\rho p t} 
			\exp\Big\{p\beta_2  W_{2,t}-\frac{1}{2}p\beta_2^2t\Big\}
			\left(
			h_0^p+ t^{p-1}
			|\rho Rk_0|^p e^{p(B+\rho)t}\int_0^t 
			\frac{\mathcal{E}^p(\beta_1W_{1,s})}
			{\mathcal{E}^p(\beta_2W_{2,s})}
			\,\de s
			\right)\notag
			\\&
			\leq  C\bigg(1+
			e^{\frac{1}{2}p (p-1)\beta_2^2t}
			\mathcal{E}(p\beta_2  W_{2,t})
			+{t^{2p-1}}
			e^{2p(B+\rho)t}\int_0^t 
			\frac{\mathcal{E}^{2p}(\beta_1W_{1,s})}
			{\mathcal{E}^{2p}(\beta_2W_{2,s})}
			\,\de s
			\bigg),\label{persport}
		\end{align}
		where the second step follows also from Young's inequality. Taking expectations,  the independence of $W_1$ and
		$W_2$ implies that 
		\begin{align}\label{est_hpositive}
			\mathbb{E}[h_t^p]
			\leq
			C\left(1+t^{2p}\right)
			\exp\left\{
			\Big[\frac12p^2\beta_2^2+2p(\beta_2^2+B+\rho+p(\beta_1^2+\beta_2^2))\Big]t
			\right\}.
		\end{align}
		We now consider the negative $q-$moments. Since $k_t>0$, It\^o's
		formula applied to $\log k_t$ yields
		$\de
		\log k_t
		=
		\big(B-\frac{c_t}{k_t}-\frac{\beta_1^2}{2}
		\big)\de t
		+
		\beta_1\de W_{1,t}.
		$
		Then the control constraint $c_t\leq Rk_t$ implies that 
		\[
		k_t^{-q}\le k_0^{-q} \mathcal{E}(-q\beta_1W_{1,t}) 
		\exp\bigg\{
		q\left(
		R-B+\frac{q+1}{2}\beta_1^2
		\right)^+t
		\bigg\},
		\]
		hence  $\mathbb{E}[k_t^{-q}]\le k_0^{-q} 
		\exp\big\{
		q\big(
		R-B+\frac{q+1}{2}\beta_1^2
		\big)^+t
		\big\}$.
		As for $h$, since $c_t>0$, from \eqref{expl_eq_h} we have 
		$
		\mathbb{E}[h_t^{-q}]
		\leq
		h_0^{-q}
		\exp\big\{
		\big(
		q\rho+\frac{q(q+1)}{2}\beta_2^2
		\big)t
		\big\}.
		$
		Combining these two estimates  with \eqref{eq:proof_positive_k} and  \eqref{est_hpositive} gives \eqref{eq:Phi_moment_extended}.
		
		We now observe that the previous moment estimates are derived from pathwise inequalities in which the upper bound is given either by a Doléans--Dade martingale multiplied by an increasing deterministic function, or by an increasing process (see, in particular, \eqref{persport}). Therefore, the optional stopping theorem and the above computations yield \eqref{eq:Phi_moment_extended} for every stopping time $\tau$ bounded by $t$, completing the proof.
	\end{myproof}

	\bibliographystyle{plainnat}
	\bibliography{bibliography}
	
\end{document}